\documentclass{article}
\usepackage{amsmath}
\usepackage{amssymb}
\usepackage{amsfonts}
\usepackage{amscd}
\usepackage{mathrsfs}
\usepackage{amsthm}
\usepackage{natbib}

\newcommand{\CC}{{\rm\bf C}}
\newcommand{\RR}{{\rm\bf R}}
\newcommand{\QQ}{{\rm\bf Q}}
\newcommand{\ZZ}{{\rm\bf Z}}

\newcommand{\Adeles}{{\rm\bf A}}

\DeclareMathOperator{\GL}{\mathrm{GL}}
\DeclareMathOperator{\SL}{\mathrm {SL}}

\DeclareMathOperator{\Oo}{\mathrm {O}}

\DeclareMathOperator{\SO}{\mathrm {SO}}

\DeclareMathOperator{\Aut}{\mathrm {Aut}}
\DeclareMathOperator{\End}{\mathrm {End}}
\DeclareMathOperator{\Hom}{\mathrm {Hom}}

\DeclareMathOperator{\kernel}{\mathrm {ker}}
\DeclareMathOperator{\image}{\mathrm {im}}

\DeclareMathOperator{\tr}{\mathrm{tr}}

\DeclareMathOperator{\ind}{\mathrm{Ind}}
\DeclareMathOperator{\pro}{\mathrm{Pro}}
\DeclareMathOperator{\res}{\mathrm{Res}}

\DeclareMathOperator{\ad}{\mathrm{ad}}

\DeclareMathOperator{\Lie}{\mathrm{Lie}}

\newcommand{\liea}{{\mathfrak {a}}}
\newcommand{\lieb}{{\mathfrak {b}}}

\newcommand{\lieg}{{\mathfrak {g}}}
\newcommand{\lieh}{{\mathfrak {h}}}

\newcommand{\liek}{{\mathfrak {k}}}
\newcommand{\liel}{{\mathfrak {l}}}
\newcommand{\liem}{{\mathfrak {m}}}
\newcommand{\lien}{{\mathfrak {n}}}

\newcommand{\liep}{{\mathfrak {p}}}
\newcommand{\lieq}{{\mathfrak {q}}}

\newcommand{\liet}{{\mathfrak {t}}}
\newcommand{\lieu}{{\mathfrak {u}}}

\newcommand{\liegl}{{\mathfrak {gl}}}
\newcommand{\liesl}{{\mathfrak {sl}}}
\newcommand{\lieso}{{\mathfrak {so}}}

\theoremstyle{plain}
\newtheorem{theorem}{Theorem}[section]

\newtheorem{corollary}[theorem]{Corollary}
\newtheorem{proposition}[theorem]{Proposition}
\newtheorem{conjecture}[theorem]{Conjecture}

\theoremstyle{remark}

\bibpunct{[}{]}{,}{n}{}{,}

\begin{document}

{\small
\title{Algebraic Characters of Harish-Chandra Modules and Arithmeticity}
\author{Fabian Januszewski}
\maketitle
}
\begin{abstract}
These are expanded notes from lectures at the Workshop {\em Representation Theory and Applications} held at Yeditepe University, Istanbul, in honor of Roger E.\ Howe. They are supplemented by the application of algebraic character theory to the construction of Galois-equivariant characters for Harish-Chandra modules.
\end{abstract}

{\small\tableofcontents}

\section*{Introduction}

In this notes we give an introduction to abstract algebraic character theory as introduced in \cite{januszewskipreprint}. The initial motivation for this theory was the study of periods of automorphic representations with applications to number theory. In general these periods are controlled by non-admissible branching problems on the level of $(\lieg,K)$-modules, where Harish-Chandra's classical global character theory is not directly applicable. One obstacle being that the analytic definition of a character does not extend to non-admissible representations, and the other being that Harish-Chandra's correspondence between closed $G$-invariant submodules and $(\lieg,K)$-submodules fails in this generality as well. So even if there would be a global character for the restriction, it is not clear that this would still be an invariant of the underlying $(\lieg,K)$-module.

To be more concrete, consider a real reductive Lie group $G$, a closed reductive subgroup $H\subseteq G$, and a unitary representation of $G$ on a Hilbert space $V$ say. Assume that $V$ is of finite length as a $G$-representation. Then Harish-Chandra has shown that it is admissible, in the sense that for a maximal compact subgroup $K\subseteq G$, and any irreducible unitary $K$-module $W$ its multiplicity $m_W(V)$ in $V$ is finite. Furthermore the subspace of $K$-finite vectors $V^{(K)}\subseteq V$ is a module for the complexified Lie algebra $\lieg$ of $G$, i.e.\ it is a $(\lieg,K)$-module. Harish-Chandra went on to prove that there is a natural bijective one-to-one correspondence between closed $G$-invariant subspaces $U\subseteq V$ and algebraic $(\lieg,K)$-submodules $U'\subseteq V^{(K)}$, given in one direction by $U\mapsto U^{(K)}$, in the other by taking closures. This way he algebrized the study of finite length (more generally admissible) unitary representations of $G$.

Now if we consider $V$ as a representation of $H$, it need no more be of finite length, it may even fail to be admissible, and the correspondence between closed $H$-invariant subspaces and algebraic $(\lieh,L)$-submodules may fail spectacularly too. One reason being that $V^{(K)}$ is a dense $(\lieh,L)$-module in $V$, and in general $V^{(K)}\subsetneq V^{(L)}$. In particular $V^{(L)}$ may be thought of as a certain {\em completion} of $V^{(K)}$, and yet these two modules behave differently: If they are not the same, then $V^{(L)}$ contains more $(\lieh,L)$-submodules than $V^{(K)}$, for the former can distinguish between $V^{(L)}$ and $V^{(K)}\subsetneq V^{(L)}$.

As the global and the infinitesimal pictures no more coincide in general, it is natural to look for a theory which genuinely lives on the infinitesimal side. The algebraic theory discussed here provides precisely this and thus overcomes some of the above limitations by seperating it from the analytic framework.

The main idea is to apply cohomological methods, which are algebraic in nature, to define a reasonable character theory. We will motivate our construction in the first section where we also review the classical theories.

We tried to keep these notes as self-contained as possible. However the proofs of many fundamental facts about $(\lieg,K)$-modules are far too involved to be treated here. For those the reader may consult the textbooks of Knapp, Vogan, Wallach, Dixmier and others, and of course papers of Harish-Chandra. The monograph \cite{book_knappvogan1995} contains most of the fundamental material in the generality we need. For a streamlined general treatment of the algebraic character theory itself we refer to \cite{januszewskipreprint}.

In order to proof something in the first 4 sections, we discuss duality theorems in detail. We hope that this allows a reader not so familar with Lie algebra cohomology to demystify the objects, as the arguments are elementary and yet the resulting statements real theorems.

A fundamental theme we left out in our discussion here is coherent continuation. We do not discuss the (good) behavior of characters under translation functors. This is treated in \cite{januszewskipreprint}.

In the last section we introduce notions of rationality for pairs, reductive pairs and also for the corresponding modules. We show that cohomology and cohomological induction carry over to the rational setting and compare in a natural way to the classical theory over $\CC$. We show in particular that every discrete series, and more generally every cohomological representation, has a model over a number field. Furthermore this gives a nice playground for a non-trivial generalization of our algebraic character theory from section 5. Since a while the author planned to write up such a theory, and the recent work of Harris \cite{harris2012} in the context of rational Beilinson-Bernstein localization and Shimura varieties gave the final motivation to include at least the beginnings here.

While we tried to make the first 6 sections as self-contained as possible, we give less details in the last section, and also require more background, particularly about linear reductive groups and their representation theory over non-algebraically closed fields of characteristic $0$, and also Tannaka duality, as it is a convenient tool for our purposes.

The author thanks the organizers of the Workshop on Representation Theory and Applications at Yeditepe University in Istanbul for their hospitality, great organization, and good working conditions. The author thanks Roger E.\ Howe for helpful discussions and his curiosity about this theory, and the participants of the workshop for their questions. Personally the author thanks Ilhan Ikeda, Mahir B.\ Can, Safak Ozden, K\"ursat Aker, Kazim B\"uy\"ukboduk, and all the others who took very good care of me during those moving times in Istanbul.

\section{What is a character?}\label{sec:whatisacharacter}

Historically the idea of characters arose in the very same moment as representation theory itself, when Frobenius introduced and investigated characters of finite groups, and thus led the foundations of representation theory.

For a finite group $G$, and a finite-dimensional complex representation
$$
\rho_V:G\to\GL(V)
$$
the corresponding character is classically defined as the map
$$
\Theta_V:G\to\CC
$$
given by
$$
g\;\mapsto\;\tr\rho_V(g).
$$
The collection of all these functions generates a $\ZZ$-submodule $C(G)$ of the space of all functions $G\to\CC$. The elements of $C(G)$ are also called {\em virtual characters}. The collection of virtual characters is an {\em abelian group}. To the addition of characters then corresponds the direct sum of representations. As a field, $\CC$ comes with a {\em multiplication}, and functions may be multiplied and we may wonder if the product of two (virtual) characters is a (virtual) character again. This turns out to be true, as the multiplication of two characters corresponds to the tensor product of representations.

The fundamental properties of characters of finite groups then are
\begin{itemize}
\item[(A)] {\bf Additivity:} $\;\;\;\;\;\;\;\;\;\;\;\;\Theta_{V\oplus W}=\Theta_V+\Theta_W$.
\item[(M)] {\bf Multiplicativity:} $\;\;\;\Theta_{V\otimes W}=\Theta_V\cdot\Theta_W$.
\item[(I)] {\bf Independence:} Characters of pairwise distinct irreducible representations are linearly independent.
\item[(D)] {\bf Density:} The collection of $\Theta_{V}$ generates the vector space ${\rm clf}(G)$ of class functions.
\end{itemize}

Properties (A), (M), (I), and (D) have the following alternative interpretation. Consider the {\em Grothendieck group} $K(G)$ of the category of representations which is constructed as follows. First consider the set $K^+(G)$ of isomorphism classes of finite-dimensional complex representations of $G$. Then the functor
$$
(V,W)\;\mapsto\;V\oplus W
$$
induces a map
$$
+:K^+(G)\times K^+(G)\;\to\;K^+(G),
$$
which turns $K^+(G)$ into an abelian monoid, the neutral element being the isomorphism class of the tautological representation
$$
{\bf 0}:G\to\GL(0),\;\;\;g\mapsto {\bf 1}_0
$$
on the $0$-space. We turn $K^+(G)$ into an abelian group $K(G)$ by formally adjoining inverses. The result $K(G)$ is called the Grothendieck group of (finite-dimensional) representations of $G$.

It comes with a multiplication, which is induced by the functor
$$
(V,W)\;\mapsto\;V\otimes W
$$
which induces a map
$$
\cdot:\;\;\;K^+(G)\times K^+(G)\;\to\;K^+(G),
$$
and the latter multiplication law extends to all of $K(G)$. This turns $K(G)$ into a {\em commutative ring with unit}, the unit being the class of the trivial representation
$$
{\bf 1}:\;\;\;G\to \GL(\CC),\;\;\;g\mapsto 1.
$$
Now the map $V\mapsto \Theta_V$ is constant on isomorphism classes, hence induces a map
$$
\Theta:\;\;\;K(G)\to{\rm clf}(G)
$$
of the Grothendieck group to the space of class functions on $G$. Then the properties (A) and (M) are equivalent to saying that $\Theta$ is a ring homomorphism (which carries $1$ to $1$), and property (I) is equivalent to $\Theta$ being a monomorphism. Finally (D) amounts to saying that the $\CC$-span of the image of $\Theta$ is the entire space of class functions. In other words (A), (B), (I) and (D) are equivalent to the fact that the induced map
$$
\Theta:\;\;\;K(G)\otimes_\ZZ\CC\to{\rm clf}(G)
$$
is an isomorphism of $\CC$-algebras.

This notion of character depends on the notion of function on the group. There is another notion of character, which is motivated by the 1913 discovery of Elie Cartan, that the isomorphism classes of irreducible finite-dimensional representations of a complex semi-simple Lie algebra $\lieg$ are in one-to-one correspondence with {\em dominant weights}. To be more precise, fix a Borel subalgebra $\lieq\subseteq\lieg$ with Levi decomposition $\lieq=\liel+\lieu$, where $\liel$ is a Cartan subalgebra and $\lieu$ is the nilpotent radical. Then the correspondence
$$
V\;\;\;\mapsto\;\;\;V^{\lieu}:=\{v\in V\;\mid\;\forall u\in\lieu:\;u\cdot v=0\},
$$
induces a map
$$
H^0(\lieu;-):\;\;\;K(\lieg)\;\to\;K(\liel)
$$
on the level of Grothendieck groups of finite-dimensional $\lieg$- resp.\ $\liel$-modules. It is easily seen to be additive, and Elie Cartan showed that it is a monomorphism, with image the dominant weights with respect to $\lieu$. However it is not multiplicative.

A consequence of Cartan's observation is that the forgetful map
\begin{equation}
\mathcal F:\;\;\;K(\lieg)\;\to\;K(\liel),\;\;\;V\;\mapsto\;V|_\liel,
\label{eq:cartanf}
\end{equation}
is a monomorphism. By definition $\mathcal F$ is also multiplicative, in particular $\mathcal F$ is a ring homomorphism. Therefore we may consider $\mathcal F$ as a character, satisfying all the fundamental properties (A), (M) and (I), except (D). This notion of character does no more depend on functions of any sort, but it depends on the non-vanishing of Grothendieck groups.

Mere 12 years later, Hermann Weyl generalized Frobenius' work to the case of compact groups $G$, and he showed that the trace still satisfies all the properties (A), (M), (I), and property (D) reads:
\begin{itemize}
\item[] {} The span of the collection of the $\Theta_{V}$ is {\em dense} in the space ${\rm clf}^2(G)$ of $L^2$-class functions.
\end{itemize}
The latter statement is known as the (weak) Peter-Weyl Theorem. Weyl went on to prove that for a compact connected Lie group $G$ and an irreducible unitary representation $V$ of $G$ of highest weight $\lambda$, the restriction of the character to the corresponding maximal torus $T$ is explicitly given by the {\em Weyl character formula}
\begin{equation}
\Theta_V|_T\;=\;
\frac{
\sum_{w\in W(G,T)}
(-1)^{\ell(w)}
e^{w(\lambda+\rho(\lieu))-\rho(\lieu)}
}
{
\sum_{w\in W(G,T)}
(-1)^{\ell(w)}
e^{w(\rho(\lieu))-\rho(\lieu)}},
\label{eq:weylcharacterformula}
\end{equation}
where $\rho(\lieu)$ denotes the half sum of the weights in $\lieu$, $W(G,T)$ is the Weyl group and $\ell$ is the length function. As $G$ is covered by the conjugates of $T$, and as $\Theta$ is a class function, $\Theta_V|_T$ already determines $\Theta_V$ on all of $G$, analogous to Cartan's observation about $\mathcal F$ above.

Now $V$ is also a module for the complexified Lie algebra $\lieg$ of $G$, and it turns out that if we choose $\liel$ as the complexified Lie algebra of $T$, then we have the analogous identity
\begin{equation}
\mathcal F(V)\;=\;
\frac{
\sum_{w\in W(\lieg,\liel)}
(-1)^{\ell(w)}
[\lambda+\rho(\lieu))-\rho(\lieu)]
}
{
\sum_{w\in W(\lieg,\liel)}
(-1)^{\ell(w)}
[w(\rho(\lieu))-\rho(\lieu)]},
\label{eq:lieweylcharacterformula}
\end{equation}
in the localization
$$
K(\liel)[W_\lieq^{-1}],
$$
where
$$
W_\lieq\;\;\;:=
\;\;\;\sum_{w\in W(\lieg,\liel)}
(-1)^{\ell(w)}
[w(\rho(\lieu))-\rho(\lieu)].
$$
As a matter of fact the map
$$
c:\;\;\;K(\lieg)\to K(\liel)[W_\lieq^{-1}],
$$
induced by
$$
V\;\mapsto\;
\frac{
\sum\limits_{w\in W(\lieg,\liel)}
(-1)^{\ell(w)}
[\lambda+\rho(\lieu))-\rho(\lieu)]
}
{
\sum\limits_{w\in W(\lieg,\liel)}
(-1)^{\ell(w)}
[w(\rho(\lieu))-\rho(\lieu)]},
$$
is still injective, and thus may be interpreted as yet another notion of character.

Harish-Chandra managed to generalize $\Theta$ to finite length representations of real reductive groups, i.e.\ which amounts to defining distribution characters for $(\lieg,K)$-modules, and he proved an analogue of the Weyl character formula \eqref{eq:weylcharacterformula} for the discrete series \cite{harishchandra1953,harishchandra1954a,harishchandra1954b}.

More generally Assume that $X$ is an admissible representation of $G$ and that the multiplicities of the $K$-types in $X$ are bounded by a polynomial in their infinitesimal character. Then for any compactly supported $f:G\to\CC$ the operator
$$
\rho(f)\;:=\;
\int_G f(g)\cdot \rho_X(g)\, dg
$$
is of trace class and
$$
\Theta_X:\;\;\;f\;\mapsto\;{\rm tr}(\rho(f))
$$
defines a distribution on $G$. This is Harish-Chandra's global character as defined in loc.\ cit..

However there is no direct way to define $\Theta$ in more general contexts, for example for non-admissible modules: The operators appearing in the analytic definition are no more of trace class. Similarly in the infinite-dimensional setting $\mathcal F$ still makes sense for Verma-modules, but for $(\lieg,K)$-modules this notion is no more meaningful either.

It turns out that the only notion surving is the map $c$. Of course for this purpose $c$ has to be defined without falling back to $\mathcal F$ or $\Theta$. This is possible, and the motivation stems from yet another incarnation of the Weyl character formulae \eqref{eq:weylcharacterformula} and \eqref{eq:lieweylcharacterformula} given by Bertram Kostant \cite{kostant1961} in 1961:
\begin{equation}
c(V)\;=\;
\frac{
\sum\limits_{q\in\ZZ}
(-1)^q
[H^q(\lieu;V)]
}
{
\sum\limits_{q\in\ZZ}
(-1)^q
[H^q(\lieu;{\bf 1})]}.
\label{eq:kostantweylcharacterformula}
\end{equation}
Here the {\em Lie algebra cohomology} $H^q(\lieu;-)$ is the $q$-th right derived functor of the functor $H^0(\lieu;-)$ considered by Elie Cartan. It vanishes for $q<0$ and $q>\dim\lieu$, thus the above sums are indeed finite.

There are several advantages of applying homological methods. First of all cohomology is well defined for {\em every} $(\lieg,K)$-module $V$, which in principle allows us to write down the expression \eqref{eq:kostantweylcharacterformula} for {\em any} $(\lieg,K)$-module $V$. A crucial limitation is that we still need a meaningful ambient Grothendieck group. Grothendieck groups may be defined for every essentially small abelian category, but they tend to collapse once multiplicities are infinite. However this isn't always a bad thing, as things may sometimes be arranged such that some part of the group collapses, but another doesn't, and the study of the latter may even be simplified by the annihilation of the other. We will encounter similar phenomena related to localization, and it even turns out that interestingly vanishing still tells us something.

Another reason why cohomological methods work out here is the extended formalism homological algebra provides. Lie algebra cohomology is very well behaved, and we will see that the fundamental desirable properties (A), (M), and partially (I), all correspond to fundamental properties of cohomology. We will also see other fundamental properties that cohomologically defined characters possess, which also have their classical analogues, yet that we refrained from adding to the above list in order not to obscure the picture.

Harish-Chandra's global character theory works well for finite length modules, as those are small \lq{}by definition\rq{} and also by Harish-Chandra's Admissibility Theorem, which makes the analytic definition of the distribution character work. If we apply our theory to the category of finite length modules, we essentially recover Harish-Chandra's characters, and his results tell us that localization does no harm here.

One fundamental problem then is how to define \lq{}small\rq{} categories. Motivated by our study of localization, we give definitions which are motivated by hypothetical generalizations of Harish-Chandra's Admissibility Theorem for more general branching problems. We will come back to this in section 6.

\section{Reductive pairs and $(\lieg,K)$-modules}

In this section we introduce the fundamental notions and properties of Harish-Chandra modules. In the literature the term {\em Harish-Chandra module} appears in different variations. The common definitions differ slightly by the finiteness conditions (finitely generatedness, admissibility, ...) imposed. Therefore we prefer the term $(\lieg,K)$-module, which a priori imposes no finiteness conditions at all (except local $K$-finiteness). We use the expression Harish-Chandra module only informally, and understand it synonymously for $(\lieg,K)$-module.

\subsection{Reductive pairs}

Let $G$ be a reductive Lie group with finite component group. Then all maximal compact subgroups in $G$ are conjugate and share the same component group with $G$. We fix one maximal compact subgroup $K\subseteq G$, and write write $\lieg_0$ for the Lie algebra of the connected component $G^0\subseteq G$ and denote $\lieg=\CC\otimes_\RR\lieg_0$ its complexification. Then $(\lieg,K)$ is a reductive pair, and there is a natural dictionary between reductive pairs and reductive Lie groups $G$ as above.

Strictly speaking a reductive pair consists of more data than the notation suggests: The map $\lieg_0\to\lieg$ is part of the datum, the extension of the adjoint action of $K$ on $\liek_0$ to $\lieg_0$ is, as is the Cartan involution $\theta:\lieg_0\to\lieg_0$, inducing a Cartan decomposition
$$
\lieg_0\;=\;\liep_0\oplus\liek_0
$$
of the Lie algebra $\lieg_0$ into $(-1)$- resp.\ $1$-eigen spaces ($\liek_0$ being identified with ${\rm Lie}(K^0)$), and the non-degenerate symmetric bilinear form $\langle\cdot,\cdot\rangle$ on $\lieg_0$ generalizing the Killing form.

Then the map
\begin{equation}
\liep_0\times K\;\to\;G,\;\;\;(p,k)\mapsto \exp(p)\cdot k
\label{eq:Gexp}
\end{equation}
is a diffeomorphism. Using this map as a blueprint, one may eventually reconstruct $G$ as a reductive Lie group entirely starting from the information provided by the reductive pair $(\lieg,K)$. For details the reader may consult \cite[Chap.\ IV]{book_knappvogan1995}.

As an example, consider the reductive Lie group $\GL_n(\RR)$. Its complexified Lie algebra $\liegl_n$ consists of all complex $n\times n$-matrices, and $\liegl_{n,0}$ consists of the (real) subalgebra of real $n\times n$ matrices. The Lie bracket is the usual commutator bracket. A maximal compact subgroup in $\GL_n(\RR)$ is the group $\Oo(n)$ of orthogonal matrices, i.e.\ real matrices $A$ satisfying $A\cdot A^t={\bf1}_n$. Then $\Oo(n)$ acts on $\liegl_n$ and $\liegl_{n,0}$ via conjugation and $(\liegl_n,\Oo(n))$ is the reductive pair corresponding to $\GL_n(\RR)$. Here $\theta:\liegl_n\to\liegl_n$ is the map $g\mapsto -g^t$. If we are interested in the connected group $\GL_n(\RR)^0$ we obtain correpondingly $(\liegl_n,\SO(n))$.

Similarly the group $\SL_n(\RR)$ corresponds to the pair $(\liesl_n,\SO(n))$, where $\liesl_n$ denotes all complex $n\times n$-matrices of trace $0$.

\subsection{Modules for pairs}

Reductive pairs are a motivation for the following more general notion of {\em pair}. A pair $(\liea,B)$ consists of a finite-dimensional complex Lie algebra $\liea$, a compact Lie group $B$, a monomorphism of complex Lie algebras
$$
\CC\otimes_\CC{\rm Lie}(B^0)=:\lieb\;\to\;\liea,
$$
and an extension of the adjoint action of $B$ on $\lieb$ to all of $\liea$, whose differential is the adjoint action of $\lieb$ on $\liea$. Then an $(\liea,B)$-module is a complex vector space $X$ with actions of both $\liea$ and $B$ subject to the following conditions:
\begin{itemize}
\item[($H_1$)] The actions of $\liea$ and $B$ on $X$ are compatible:
$$
\forall a\in\liea,\;b\in B,\;x\in X:\;\;\;b\cdot a\cdot b^{-1}\cdot x\;=\;[{\rm Ad}(b)a]\cdot x.
$$
\item[($H_2$)] $X$ is {\em locally $B$-finite}:
$$
\forall x\in X:\;\;\;\dim_\CC\langle B\cdot x\rangle_\CC\;<\;\infty
$$
and with respect to the (unique) natural topology on $\langle B\cdot x\rangle_\CC$ the representation
$$
B\to\Aut_\CC(\langle B\cdot x\rangle_\CC)
$$
is continuous.
\item[($H_3$)] The differential of the action of $B$ is the action of $\lieb\subseteq\liea$ on $X$, i.e.\ 
$$
\forall b\in\lieb_0,\; x\in X:\;\;\;b\cdot x\;=\;\left[\frac{d}{dt}\frac{\exp(tb)\cdot x-x}{t}\right]_{t=0}.
$$
\end{itemize}
In this generality the definition is due to Lepowsky. We remark that ($H_3$) is meaningful, as axiom ($H_2$) implies that the action of $B$ is smooth on the subrepresentation generated by $x\in X$.

A map $X\to Y$ of $(\liea,B)$-modules is a $\CC$-linear map which is compatible with the actions of $\liea$ and $B$ in the obvious way. In particular we obtain the category $\mathcal C(\liea,B)$ of all $(\liea,B)$-modules.

An important consequence of the local $K$-finiteness is
\begin{proposition}\label{prop:kdecomp}
Let $B$ be any compact Lie group with complexified Lie algebra $\lieb$. Then every irreducible $(\lieb,B)$-module $Y$ is finite-dimensional and every $(\lieb,B)$-module $X$ is of the form
$$
X\;\cong\;\bigoplus_{Y\in\widehat{B}} X_Y,
$$
where $Y$ runs through the irreducibles and $X_Y$ denotes the $Y$-isotypic subspace in $X$.
\end{proposition}

\begin{corollary}
The categories $\mathcal C_{\rm fd}(\lieb,B)$ and $\mathcal C_{\rm fd}(B)$ of finite-dimensional $(\lieb,B)$-modules resp.\ finite-dimensional continuous $B$-representations are naturally equivalent.
\end{corollary}

The same remains true for reductive pairs.

\begin{proposition}
Let $(\lieg,K)$ be a reductive pair and $G$ the corresponding real reductive Lie group. Then the categories of finite-dimensional $(\lieg,K)$-modules and finite-dimensional continuous $G$-representations are equivalent.
\end{proposition}

\begin{proof}
The hard part is to show that every continuous representation of $G$ is smooth. This follows from the fact that every continuous group homomorphisms of Lie groups is smooth. Then we apply this to the continuous group homomorphism $G\to \GL(V)$ corresponding to our representation. The finite-dimensionality of $V$ is crucial, as it guarantees that $\GL(V)$ is a Lie group. Therefore $V$ is also a $(\lieg,K)$-module. The other way around, departing from a $(\lieg,K)$-module $V$, we'd like to lift the representation $\rho:K\to\GL(V)$ (uniquely) to $G$. That this is indeed possible follows from \eqref{eq:Gexp}, which essentially tells us that $\pi_0(G)=\pi_0(K)$ and $\pi_1(G)=\pi_1(K)$, i.e.\ all topological obstructions are already taken care of by $K$. More explicitly we may define
$$
\rho_V(\exp(p)\cdot k)\;:=\;\exp(\rho(p))\cdot\rho(k),
$$
for $p\in\liep_0$ and $k\in K$. This is the unique extension to $G$. Again the finite-dimensionality is crucial, as it guarantees the existence of $\exp(\rho(p))\in\GL(V)$.
\end{proof}

\subsection{Internal constructions}

For any two $(\liea,B)$-modules $X$ and $Y$ the tensor product $X\otimes_\CC Y$ acquires a natural action of the pair $(\liea\times\liea,B\times B)$, which is explicitly given by
$$
(a_1,a_2)\cdot (x\otimes y)\;=\;(a_1\cdot x)\otimes y\;+\;x\otimes(a_2\cdot y)
$$
for $a_1,a_2\in\liea$ and $x\in X$, $y\in Y$. For $b_1,b_2\in B$ we have analogously
$$
(b_1,b_2)\cdot (x\otimes y)\;=\;(b_1\cdot x)\otimes (b_2\cdot y).
$$
Then the pullback of this action along the diagonal map
$$
\Delta:\;\;\;(\liea,B)\;\to\;(\liea\times\liea,B\times B)
$$
$$
(a,b)\;\mapsto\;((a,a),(b,b))
$$
turns $X\otimes_\CC Y$ into an $(\liea,B)$-module.

Similarly we may consider the space $\Hom_\CC(X,Y)$ of linear maps $f:X\to Y$. This acquires an action of $(\liea\times\liea,B\times B)$ which is given by
$$
[(a_1,a_2)\cdot f](x)\;=\;a_2\cdot f(-a_1\cdot x)
$$
and similarly
$$
[(b_1,b_2)\cdot f](x)\;=\;b_2\cdot f(b_1^{-1}\cdot x).
$$
However in general this action is not locally $B\times B$-finite. Therefore we are obliged to pass to the subspace
$$
\Hom_\CC(X,Y)_{B\times B}
$$
of locally $B\times B$-finite vectors. This then is an $(\liea\times\liea,B\times B)$-module. In order to obtain an $(\liea,B)$-module, we may again consider the pullback along the diagonal $\Delta$. Yet in this case there is a subtelty, as in general we obtain the desired $(\liea,B)$-module only as the $B$-finite subspace
$$
\Hom_\CC(X,Y)_{B}\;=\;\Hom_\CC(X,Y)_{\Delta(B)}
$$
of $\Hom_\CC(X,Y)$, as the latter is in general strictly bigger than the subspace of $B\times B$-finite vectors. However if $X$ is finite-dimensional then all these spaces aggree with $\Hom_\CC(X,Y)$.

We define the {\em dual} of $X$ as
$$
X^\vee\;:=\;\Hom_\CC(X,{\bf1})_B,
$$
where ${\bf1}$ denotes the trivial $(\liea,B)$-module which is isomorphic to $\CC$ as a vector space.

As a $B$-module we may think of $X^\vee$ as the direct sum of the locally $B$-finite duals $(X_Y)^\vee$ of the isotypic components $X_Y$ in the sense of Proposition \ref{prop:kdecomp}. In particular if all $X_Y$ are finite-dimensional, then $X$ is {\em reflexive}, i.e.\ the canonical bidual map $X\to X^{\vee\vee}$ is an {\em isomorphism} in this case.

We have a natural monomorphism
\begin{equation}
\psi_{X,Y}:\;\;\;X^\vee\otimes_\CC Y\;\to\;\Hom_\CC(X,Y)_B,
\label{eq:tensorhom}
\end{equation}
$$
\xi\otimes y\;\;\mapsto\;\;[x\;\mapsto\;\xi(x)\cdot y]
$$
of $(\liea,B)$-modules. This is an isomorphism whenever $X$ is reflexive.



\subsection{The Harish-Chandra map and infinitesimal characters}

We write $Z(\lieg)$ for the center of the universal enveloping algebra $U(\lieg)$ of $\lieg$. We say that a $(\lieg,K)$-module $X$ has an {\em infinitesimal character}, if $Z(\lieg)$ acts via scalars in $X$. Then the character
$$
\chi:\;\;\;Z(\lieg)\to\CC
$$
defined by this action is called the {\em infinitesimal character} of $X$. It is characterized by the identity
$$
z\cdot x\;=\;\chi(z) x
$$
for all $z\in Z(\lieg)$ and all $x\in X$.

\begin{proposition}[Dixmier]\label{prop:dixmier}
For any irreducible $(\lieg,K)$-module $X$ which remains irreducible as a $U(\lieg)$-module we have
$$
\End_{\lieg,K}(X)=\CC\cdot{\rm id}_X.
$$
In particular such an $X$ has an infinitesimal character.
\end{proposition}

Proposition \ref{prop:dixmier} is a generalization of Schur's Lemma. For a proof see \cite{dixmier1963}, and also \cite[Proposition 2.6.8]{book_dixmier1977} or \cite[Proposition 4.87]{book_knappvogan1995}.

The $U(\lieg)$-irreducibility is automatically satisfied whenever $K$ is connected and $X$ irreducible as a $(\lieg,K)$-module.

Fix a Borel subalgebra $\lieb\subseteq\lieg$, with Levi decomposition
$$
\lieb\;=\;\lieh+\lieu.
$$
Then the Poincar\'e-Birkhoff-Witt Theorem tells us that we have a direct sum decomposition
$$
U(\lieg)\;\;=\;\;U(\lieh)\;\oplus\;(\lieu^- U(\lieg)+U(\lieg)\lieu),
$$
and we denote the projection onto the first summand by $p_\lieu$. We write
$$
\rho(\lieu)\;:=\;\frac{1}{2}\sum_{\alpha\in\Delta(\lieu,\lieh)}\alpha\;\in\;\lieh^*.
$$
The map
$$
h\;\mapsto\;h-[\rho(\lieu)](h)\cdot 1_{U(\lieh)}
$$
for $h\in\lieh$ turns out to be a homomorphism of Lie-algebras, and therefore extends to an algebra homomorphism
$$
\rho_\lieu:\;\;\;U(\lieh)\to U(\lieh)
$$
by universality. Finally we set
$$
\gamma_\lieu\;:=\;\rho_\lieu\circ p_\lieu:\;\;\;U(\lieg)\to U(\lieh).
$$
\begin{theorem}[Harish-Chandra]\label{thm:harishchandraisomorphism}
The map $\gamma_\lieu$ induces an algebra isomorphism
$$
\gamma:\;\;\;Z(\lieg)\to U(\lieh)^{W(\lieg,\lieh)},
$$
which only depends on $\lieh$, i.e.\ is independentent of the choice of $\lieu$.
\end{theorem}

The map $\gamma$ is called {\em Harish-Chandra isomorphism} and enables us to describe infinitesimal characters via $W(\lieg,\lieh)$-orbits in $\lieh^*$. It turns out that the correspondence
$$
W(\lieg,\lieh)\backslash{}\lieh^*\;\to\;\Hom_{\rm alg}(Z(\lieg),\CC),
$$
$$
W(\lieg,\lieh)\lambda\;\mapsto\;\chi_\lambda:=\lambda\circ\gamma
$$
is bijective, where we implicitly used the universal property of $U(\lieh)$ on the right hand side. Via this correspondence we may say that a $(\lieg,K)$-module $X$ has infinitesimial character $\lambda\in\lieh^*$, if its infinitesimal character coincides with $\chi_\lambda$.

Infinitesimal characters are strong invariants. An easy excercise shows that the infinitesimal character of an irreducible finite-dimensional $\lieg$-module of highest weight $\lambda$ is $\chi_{\lambda+\rho(\lieu)}$. Consequently two irreducible finite-dimensional $\lieg$-modules are isomorphic if and only if their infinitesimal characters coincide.

In general we have the fundamental
\begin{theorem}[Harish-Chandra]\label{thm:harishchandrafinite}
Assume that $K$ is connected. For any fixed $\lambda\in\lieh^*$ there are only finitely many isomorphism classes of irreducible $(\lieg,K)$-modules with infinitesimal character $\lambda$.
\end{theorem}

The classical proof of this theorem relies on Harish-Chandra's global characters. By algebraic methods one can show that, for any fixed $\lambda$, there are finitely many $K$-type that each irreducible $(\lieg,K)$-module with infinitesimal character $\lambda$ must contain, cf.\ \cite[Theorem 7.204]{book_knappvogan1995}. Then David Vogan's minimal $K$-type theory concludes the proof.

We see in Theorem \ref{thm:harishchandrafinite} that the condition on $X$ of having an infinitesimal character is very strict. We relax it as follows. We say that a $(\lieg,K)$-module $X$ is {\em $Z(\lieg)$-finite}, if the annihilator of $X$ in $Z(\lieg)$ is of finite codimension. Any $Z(\lieg)$-finite module may be decomposed into a direct sum of its $Z(\lieg)$-primary components, which all are $(\lieg,K)$-submodules.

\subsection{Composition factors and multiplicities}

Fix a reductive pair $(\lieg,K)$ and a $(\lieg,K)$-module $X$. We say that $X$ has an irreducible $(\lieg,K)$-module $Y$ as a {\em composition factor} if there are submodules $X_1\subseteq X_0\subseteq X$ such that $X_1/X_0\cong Y$. We write $S(X)$ for the set of submodules of $X$. It is comes with a natural preorder given by set-theoretic inclusion.

The {\em multiplicity} of $Y$ in $X$ is the supremum $m_Y(X)$ of the cardinalities of all totally ordered sets $(I,\leq)$ with the property that there exist injective order-preserving maps $a:I\to S(X)$ and $b:I\to S(X)$ such that for any $i\in I$ we have $a(i)\subseteq b(i)$ and $b(i)/a(i)\cong Y$.

The multiplicity of $Y$ in $X$ is finite if and only if the set $\mu_Y(X)$ of natural numbers $m$ with the property that there exists a natural number $N$ and submodules
$$
X_0\subseteq X_1\subseteq\cdots\subseteq X_N\subseteq X
$$
such that $X_i/X_{i+1}\cong Y$ for $m$ distinct indices $0\leq i<N$ is bounded. If this is the case, then $m_Y(X)=\max\mu_Y(X)$.

Suppose we are given a short exact sequence
$$
0\to A\to B\to C\to 0
$$
of $(\lieg,K)$-modules. Then
\begin{equation}
m_Y(A)+m_Y(C)\;=\;m_Y(B),
\label{eq:multiplicityaddition}
\end{equation}
where \lq$+$\rq{} denotes addition of cardinal numbers.

We say that two $(\lieg,K)$-modules $X$ and $X'$ have the same {\em semi-simplifications} if for all irreducible $Y$ we have $m_Y(X)=m_{Y}(X')$. Proposition \ref{prop:kdecomp} tells us that if $\lieg=\liek$ then the isomorphism class of $X$ depends only on its semi-simplification. However we emphasize that in general a non-zero $(\lieg,K)$-module $X$ may possess no composition factor at all, and also non-trivial extension classes between irreducibles may exist. Consequently semi-simplification is far from being a faithful operation.

\subsection{Admissible and finite length modules}

We call $X$ {\em admissible} if, as a $K$-module, all multiplicities of irreducibles in $X$ are finite. For an admissible $X$ the multiplicities $m_Y(X)$ are necessarily finite for any irreducible $(\lieg,K)$-module $Y$.

We say that $X$ is of {\em finite length} if $X$ has a finite composition series.
\begin{proposition}\label{prop:finitelength}
Let $X$ be a $(\lieg,K)$-module. The following statements are equivalent:
\begin{itemize}
\item[(i)] $X$ is of finite length.
\item[(ii)] $X$ is admissible and finitely generated.
\item[(iii)] $X$ is admissible and $Z(\lieg)$-finite.
\end{itemize}
\end{proposition}

\begin{proof}
If $X$ is an irreducible $(\lieg,K)$-module, then it is admissible \cite{lepowsky1973}. As a module of finite length is finitely generated, this shows that (i) implies (ii).

By Dixmier's Proposition \ref{prop:dixmier}, an irreducible $X$ has an infinitesimal character for $K$ connected, hence (i) implies (iii).

The implication (ii) $\Longrightarrow$ (iii) is standard. The remaining implication (iii) $\Longrightarrow$ (i), may be deduced from Theorem \ref{thm:harishchandrafinite}, using that $U(\lieg)$ is noetherian. See the proof of Corollary 7.207 in \cite{book_knappvogan1995} for example.
\end{proof}

\subsection{Discretely decomposable modules}

Following Kobayashi \cite[Definition 1.1]{kobayashi1997}, we say that $X$ is {\em discretely decomposable} if $X$ is a union of finite length modules, and a discretely decomposable $X$ is {\em discretely decomposable with finite multiplicities} if all $m_Y(X)$ are finite.

By Proposition \ref{prop:kdecomp} an admissible $X$ is discretely decomposable with finite multiplicities as $(\liek,K)$-module. Then the multiplicity of any irreducible $(\lieg,K)$-module $Y$ in $X$ is finite. We remark however $X$ need not be discretely decomposable. The other way round we have the following criterion.

\begin{proposition}[{Kobayashi, \cite[Lemma 1.5]{kobayashi1997}}]
Let $(\lieh,L)\to(\lieg,K)$ be an inclusion of reductive pairs. For any irreducible $(\lieg,K)$-module $X$ the following are equivalent:
\begin{itemize}
\item[(i)] $X$ is a discretely decomposable $(\lieh,L)$-module.
\item[(ii)] It exists a finite length $(\lieh,L)$-module $Z$ and a non-zero $(\lieh,L)$-map $Z\to X$.
\item[(iii)] It exists an irreducible $(\lieh,L)$-module $Y$ and a non-zero $(\lieh,L)$-map $Y\to X$.
\end{itemize}
\end{proposition}

\subsection{Categories of $(\lieg,K)$-modules}

In order to make our character theory work, we need essentially small full abelian subcategories of $\mathcal C(\lieg,K)$ (for the notion of an {\em abelian category} we refer to section \ref{sec:grothendieckgroups} below). The properties from the previous section all define certain abelian subcategories, some of which provide nice setups for our theory. For this purpose we write
${\mathcal C}_{\rm a}(\lieg,K)$ resp.\ 
${\mathcal C}_{\rm d}(\lieg,K)$ resp.\ 
${\mathcal C}_{\rm df}(\lieg,K)$ resp.\ 
${\mathcal C}_{\rm zf}(\lieg,K)$ resp.\ 
${\mathcal C}_{\rm fl}(\lieg,K)$ resp.\ 
${\mathcal C}_{\rm fd}(\lieg,K)$
for the categories of admissible resp.\ 
discretely decomposable resp.\ 
discretely decomposable with finite multiplicities resp.\ 
$Z(\lieg)$-finite resp.\ 
finite length resp.\ 
finite-dimensional $(\lieg,K)$-modules. They satisfy the inclusion relations
$$
{\mathcal C}(\lieg,K)\;\supset\;
{\mathcal C}_{\rm d}(\lieg,K)\;\supset\;
{\mathcal C}_{\rm df}(\lieg,K)\;\supset\;
{\mathcal C}_{\rm fl}(\lieg,K)\;\supset\;
{\mathcal C}_{\rm fd}(\lieg,K),
$$
$$
{\mathcal C}(\lieg,K)\;\supset\;
{\mathcal C}_{\rm zf}(\lieg,K)\;\supset\;
{\mathcal C}_{\rm fl}(\lieg,K)\;\supset\;
{\mathcal C}_{\rm fd}(\lieg,K),
$$
and
$$
{\mathcal C}(\lieg,K)\;\supset\;
{\mathcal C}_{\rm a}(\lieg,K)\;\supset\;
{\mathcal C}_{\rm fl}(\lieg,K)\;\supset\;
{\mathcal C}_{\rm fd}(\lieg,K).
$$
Proposition \ref{prop:finitelength} tells us that
$$
{\mathcal C}_{\rm zf}(\lieg,K)\;\cap\;
{\mathcal C}_{\rm a}(\lieg,K)\;=
{\mathcal C}_{\rm fl}(\lieg,K).
$$
The above categories are in general not closed under tensor products. For our purpose we have the important
\begin{proposition}\label{prop:tensorstable}
Fix $?\in\{\rm d,\rm df,\rm zf, \rm a,\rm fl,\rm fd\}$. For any $Z\in{\mathcal C}_{\rm fd}(\lieg,K)$ and any $X\in{\mathcal C}_{?}(\lieg,K)$ we have
$$
Z\otimes_\CC X\;\in\;{\mathcal C}_{?}(\lieg,K).
$$
\end{proposition}

\begin{proof}
We only sketch a proof for $?={\rm df}$ and otherwise give references to the literature. The case $?=\rm fd$ is clear. The crucial case is $?={\rm zf}$, as it allows one to invoke Proposition \ref{prop:finitelength}. A treatment of this case may be found in \cite[Proposition 7.203]{book_knappvogan1995}.

The case $?=\rm fl$ is due to Kostant \cite{kostant1975}. To prove it one can deduce it from $?={\rm zf}$ using Proposition \ref{prop:finitelength}, and \cite[Corollary 4.5.6]{book_vogan1981} also contains a proof. The case $?=\rm d$ follows from the fact that the endofunctor $Z\otimes_\CC -$ of ${\mathcal C}(\lieg,K)$ commutes with direct limits, and was first observed by Kobayashi \cite[Lemma 1.4]{kobayashi1997}.

We assume for a moment that $K$ is connected. Then the case $?={\rm df}$ follows from Kostant's observation that tensoring with $Z$ \lq{}shifts\rq{} infinitesimal characters only in a finite manner, namely by weights occuring in $Z$, cf.\ loc.\ cit.\ and also \cite[Theorem 7.133]{book_knappvogan1995}, \cite[Lemma 4.5.4 and Corollary 4.5.6]{book_vogan1981}. In particular, if $Y$ is an irreducible $(\lieg,K)$-module, then there are (up to isomorphy) only finitely many irreducible $Y'$ with $m_Y(Z\otimes_\CC Y')\neq 0$ due to Harish-Chandra's Theorem \ref{thm:harishchandrafinite}.

Consequently, if we write
$$
X=\bigcup\limits_{i\in I}X_i
$$
with each $X_i$, $i \in I$ of finite length, then as the multiplicity of each $Y'$ in $X$ is finite, there are only finitely many $i \in I$ with $m_Y(Z\otimes X_i)\neq 0$, which reduces us to the case $?={\rm fl}$.

The case of non-connected $K$ then follows from the observation that each finite-length $(\lieg,K)$-module $E$ is a finite length $(\lieg,K^0)$-module, and if all multiplicities for $(\lieg,K)$ are finite, the same is true in the connected case, and vice versa.
\end{proof}

\section{$K$-groups and their localizations}

\subsection{Grothendieck groups}\label{sec:grothendieckgroups}

Consider an {\em abelian category} $\mathcal A$, i.e.\ a category in which all morphism-sets are abelian groups in which the composition of morphisms is $\ZZ$-bilinear, kernels and cokernels exist (i.e.\ $\mathcal A$ is {\em additive}), and every monomorphism is a kernel and every epimorphism is a cokernel. Evidently all the categories ${\mathcal C}_?(\lieg,K)$ are abelian.

In an abelian category we have the notion of {\em exact sequence}: A sequence of morphisms
\begin{equation}
\begin{CD}
X@>\alpha>> Y@>\beta>> Z
\end{CD}
\label{eq:xyz}
\end{equation}
is said to be a {\em exact at $Y$} if the kernel of $\beta$ agrees with the image of $\alpha$. The notion of image has an abstract definition in an abelian category: it is the kernel of the cokernel. A {\em short exact sequence} is a sequence exact at $Y$ as above where additionally $\alpha$ is a monomorphism and $\beta$ is an epimorphism.

We call $\mathcal A$ {\em essentially small} if it is equivalent to a category whose class of objects is a set.

For an essentially small abelian category $\mathcal A$ we may consider the set $A$ of isomorphism classes of objects in $\mathcal A$. Then in the free abelian group $\ZZ[A]$ over $A$ the collection of elements
$$
Y-X-Z\;\in\;\ZZ[A]
$$
which occur in a short exact sequence \eqref{eq:xyz} generate a subgroup $R\leq\ZZ[A]$, and we define the {\em Grothendieck group} of $\mathcal A$ as the abelian group
$$
K(\mathcal A)\;:=\;\ZZ[A]/R.
$$
It comes with a natural map
$$
[\cdot]:\;\;\;\mathcal A\;\to\;K(\mathcal A),
$$
which is induced by sending an object $X$ first to its isomorphism class in $A$, then to its canonical image in $\ZZ[A]$, and then project it modulo $R$.

The map $[\cdot]$ is easily seen to be surjective and to satisfy the following universal property: For any abelian group $B$ and any {\em additive} map $f:\mathcal A\to B$, i.e.\ a map for which
$$
f(X)+f(Z)\;=\;f(Y)
$$
for any short exact sequence \eqref{eq:xyz}, there is a unique homomorphism $f':K(\mathcal A)\to B$ with $f=f'\circ [\cdot]$.

Due to this universal property, all characters considered in section \ref{sec:whatisacharacter} factor over (i.e.\ extend to) the corresponding Grothendieck groups.

\subsection{Grothendieck groups of $(\lieg,K)$-modules}

Unfortunately the categories ${\mathcal C}_?(\lieg,K)$ for $?\in\{-,\rm zf,\rm d\}$ are {\em not} essentially small, so we cannot consider their Grothendieck groups directly.

As there are only countably many isomorphism classes of irreducible finite-dimensional representations of $(\lieg,K)$ for semisimple $\lieg$, and $\#\CC$ many for a non-compact torus, we see that for an arbitrary reductive pair $(\lieg,K)$ the category ${\mathcal C}_{\rm fd}(\lieg,K)$ is essentially small, and $K({\mathcal C}_{\rm fd}(\lieg,K))$ is the free abelian group generated by the isomorphism classes of irreducible finite-dimensional modules.

There are at most countably many irreducible representations of $K$ up to isomorphy, and therefore ${\mathcal C}_{\rm a}(\lieg,K)$ is essentially small. In this case there is no simple characterization of its Grothendieck group. However it comes with a group homomorphism
$$
{\rm res}_{\liek,K}^{\lieg,K}:\;\;\;
K({\mathcal C}_{\rm a}(\lieg,K))\;\to\; K({\mathcal C}_{\rm a}(\liek,K))
$$
induced by restriction. If $K$ is infinite, the group on the right hand side is non-canonically isomorphic to the additive group of the ring of formal power series
$
\ZZ[[T]]
$
in one indeterminate, by identifying each power of $T$ with an isomorphism class of an irreducible $K$-module.

Theorem \ref{thm:harishchandrafinite} tells us that in the remaining cases $?\in\{\rm df,\rm fl\}$ the corresponding category again is essentially small. If we write $A_0\subseteq A$ for the set of isomorphism classes of irreducibles, then as in the previous example we have a canonical isomorphism
$$
K({\mathcal C}_{\rm df}(\lieg,K))\;\cong\; \ZZ^{A_0},
$$
where the right hand side denotes the abelian group of all maps $A_0\to\ZZ$. The inclusion
$$
K({\mathcal C}_{\rm fl}(\lieg,K))\;\to\;
K({\mathcal C}_{\rm df}(\lieg,K))
$$
induces a canonical isomorphism of abelian groups
$$
K({\mathcal C}_{\rm fl}(\lieg,K))\;\cong\; \ZZ^{(A_0)},
$$
where this time the right hand side denotes the abelian group of maps $A_0\to\ZZ$ with {\em finite support}, i.e.\ that are non-zero only for finitely many arguments.

For readability we introduce the abbreviations
$$
K_?(\lieg,K)\;:=\;K({\mathcal C}_?(\lieg,K)).
$$

\subsection{The multiplicative structure}

The tensor product in ${\mathcal C}(\lieg,K)$ is exact, and consequently it descends to Grothendieck groups whenever it preserves the underlying essentially small category. However this is difficult to guarantee in general in most cases.

The situation is considerably simpler if we consider finite-dimensional modules. The category ${\mathcal C}_{\rm fd}(\lieg,K)$ is closed under tensor products, and therefore we get an induced multiplication map
$$
\cdot:\;\;\;K_{\rm fd}(\lieg,K)\times K_{\rm fd}(\lieg,K)\;\to\;K_{\rm fd}(\lieg,K),
$$
$$
([X],[Y])\;\mapsto\;[X\otimes_\CC Y].
$$
As the tensor product is associative, behaves well in exact sequences, and in particular is distributive with respect to direct sums, this turns $K_{\rm fd}(\lieg,K)$ into a commutative ring with $1$, the latter being the class of the trivial representation.

\begin{proposition}\label{prop:connecteddomain}
For connected $K$ the ring $K_{\rm fd}(\lieg,K)$ is an integral domain.
\end{proposition}

\begin{proof}
We already know that the tensor product of non-zero modules is non-zero. We have to show that the same remains true when allowing formal differences of modules. For this purpose we make use of the forgetful functor $\mathcal F$ from \eqref{eq:cartanf} from the first section which was associated to a Cartan subalgebra $\liel\subseteq\lieg$. Then we may consider a corresponding Grothendieck group $K_{\rm fd}(\liel)$ for the category of finite-dimensional $\liel$-modules. The forgetful functor $\mathcal F$ induces a ring morphism
$$
\mathcal F:\;\;\;K_{\rm fd}(\lieg,K)\;\to\; K_{\rm fd}(\liel),
$$
sending $1$ to $1$. As already discussed in the first section, E.\ Cartan's result implies that this map is injective, and consequently we are reduced to the case of $\liel$, which is an abelian Lie algebra. Here all irreducibles are one-dimensional and correspond to characters $\lambda\in\liel^*$. The tensor product of two characters is its sum. Therefore we have a canonical isomorphism
$$
K_{\rm fd}(\liel)\;\cong\; \ZZ[\liel^*]
$$
of rings, where the right hand side denotes the group algebra of $\liel^*$. The latter is obviously an integral domain, hence the claim follows.
\end{proof}

We remark that for non-connected $K$ the statement of Proposition \ref{prop:connecteddomain} is no more true. Consider the group $K=\{\pm 1\}$. The classes of irreducible representations are represented by ${\bf 1}$ and ${\rm sgn}$. In the Grotendieckgroup of finite-dimensional $K$-representations we have
$$
([{\bf 1}]+[{\rm sgn}])\cdot([{\bf 1}]-[{\rm sgn}])\;=\;
[{\bf 1}\otimes {\bf 1}]-[{\rm sgn}\otimes{\rm sgn}])\;=\;
[{\bf 1}]-[{\bf 1}]\;=\;0.
$$


If we write $F_0$ for the set of isomorphism classes of finite-dimensional $(\lieg,K)$-modules, the canonical isomorphism
$$
\iota:\;\;\;K_{\rm fd}(\lieg,K)\;\cong\; \ZZ^{(F_0)}
$$
does not respect multiplication if we define the product on the right hand side component-wise. The trick is that the tensor product induces a pairing
$$
\cdot:\;\;\;F_0\times F_0\;\to\; \ZZ^{(F_0)},
$$
which in turn induces a unique $\ZZ$-bilinear map
$$
\cdot:\;\;\;\ZZ^{(F_0)}\times\ZZ^{(F_0)}\;\to\;\ZZ^{(F_0)},
$$
which turns $\ZZ^{(F_0)}$ into a ring, and $\iota$ into a natural isomorphism of rings.

Unless $F_0$ is finite we cannot extend this product to $\ZZ^{F_0}$, as infinitely many pairs of elements of $F_0$ may contribute to the coefficient of a single one. Therefore there is no hope for larger categories to be stable under the tensor product. This already applies to ${\mathcal C}_{\rm a}(\liek,K)$.

All we can hope for is to consider the Grothendieck groups of larger categories as modules over $K_{\rm fd}(\lieg,K)$. This indeed works out in our previous examples thanks to Proposition \ref{prop:tensorstable}: For $?\in\{\rm a,\rm df,\rm fl,\rm fd\}$ the tensor product induces on $K_?(\lieg,K)$ a $K_{\rm fd}(\lieg,K)$-module structure.

\subsection{Localization of Grothendieck groups}

Assume we are given an element
$$
0\neq W\in K_{\rm fd}(\lieg,K).
$$
For connected $K$, the ring $K_{\rm fd}(\lieg,K)$ is a domain by Proposition \ref{prop:connecteddomain}, and its quotient field $Q_{\rm fd}(\lieg,K)$ comes with a monomorphism
$$
i:\;\;\;K_{\rm fd}(\lieg,K)\;\to\; Q_{\rm fd}(\lieg,K),
$$
and $i(W)$ becomes invertible. Therefore we may define the ring
$$
K_{\rm fd}(\lieg,K)[W^{-1}]\;\subseteq\;Q_{\rm fd}(\lieg,K),
$$
as the subring of the quotient field generated by the image of $i$ and $W^{-1}$.

For non-connected $K$ we define $K_{\rm fd}(\lieg,K)[W^{-1}]$ formally as the localization at $W$, i.e.\ as a set it is given by equivalence classes of pairs $(a,b)$, where $a\in K_{\rm fd}(\lieg,K)$ and $b=W^k$ for some $k\geq 0$, and we say that
$$
(a,b)\;\sim\;(a',b')\;\;\;:\Longleftrightarrow\;\;\;a\cdot b'\;=\;a'\cdot b.
$$
Then the set of equivalence classes with respect to this equivalence relation defines the desired localized ring, which comes with a natural map
$$
K_{\rm fd}(\lieg,K)\;\to\;K_{\rm fd}(\lieg,K)[W^{-1}],\;\;\;
a\;\mapsto\;[(a,1)]_\sim,
$$
which satiesfies the usual universal property.

For a $K_{\rm fd}(\lieg,K)$-module $M$ we define its localization at $W$ as
$$
M[W^{-1}]\;:=\;K_{\rm fd}(\lieg,K)[W^{-1}]\otimes_{K_{\rm fd}(\lieg,K)}M.
$$
Localization is a non-faithful operation, as it kills all $W$-torsion. More precisely we have
\begin{proposition}\label{prop:lockernel}
Let $M_W\subseteq M$ be the sum of the kernels of the endormorphisms $W^k:M\to M$ given by multiplication with $W^k$, $k\geq 1$. Then the sequence
$$
\begin{CD}
0@>>> M_W@>>> M@>>> M[W^{-1}]
\end{CD}
$$
is exact.
\end{proposition}

\begin{proof}
We may identify the elements of the image of $M$ in $M[W^{-1}]$ with equivalence classes of elements of $M$ with respect to the equivalence relation
$$
a\sim b\;\;\;:\Longleftrightarrow\;\;\;\exists k\geq 1:\;W^k\cdot a=W^k\cdot b.
$$
On the one hand, by definition, for any $m\in M_W$ there is a $k\geq 1$ with $W^k\cdot m=0$, which implies $W^k\cdot m=W^k\cdot 0$, hence $m\sim 0$, and $M_W$ maps to $0$ in $M[W^{-1}]$.

On the other hand, if $m\in M$ maps to $0$ in the localization, this means that there is a $k\geq 1$ with $W^k\cdot m=W^k\cdot 0=0$, hence $m\in M_W$.
\end{proof}

\section{Parabolic pairs and cohomology}

If $\lieq=\liel+\lieu$ is a parabolic subalgebra of $\lieg$, the functor
$$
H^0(\lieu;-):\;\;\;X\mapsto X^\lieu
$$
is very useful in the study of finite-dimensional representations. It maps $\lieg$-modules to $\liel$-modules, and the latter are usually easier understood, and therefore help in the study of the former.

In order to make this construction work for $(\lieg,K)$-modules, we need to make sure that $\liel$ is part of a reductive pair, which we want to correspond to a closed reductive subgroup $L$ of $G$, if the reductive Lie group $G$ corresponds to $(\lieg,K)$. Therefore we need to restrict our attention to {\em germane} parabolic subalgebras, which are parabolics giving rise to parabolic pairs. The latter have reductive pairs as Levi factors.

In this context $H^0(\lieu;-)$ becomes a left exact functor ${\mathcal C}(\lieg,K)\to{\mathcal C}(\liel,L\cap K)$, and its right derived functors give us for each $(\lieg,K)$-module $X$ a collection of $(\liel,L\cap K)$-modules $H^q(\lieu;X)$, $0\leq q\leq\dim\lieu$. There is a standard complex computing this cohomology. This explicit description of the cohomology will enable us later to prove fundamental properties of our algebraic characters.

For our applications to character theory we will introduce the notion of {\em constructible parabolic pairs}. Those are iteratively constructed out of real and $\theta$-stable ones and in this case $H^q(\lieu;-)$ preserves many useful properties, in particular finite length.


\subsection{Parabolic pairs}

In this section we set up the essential formalism of parabolic pairs and their Levi decompositions. For details the reader may consult \cite[Chapter IV, Section 6]{book_knappvogan1995}.

Again $(\lieg,K)$ is a reductive pair. Let $\lieq\subseteq\lieg$ be a parabolic subalgebra, and $G$ the corresponding reductive Lie group for the pair $(\lieg,K)$. We say that $\lieq$ is {\em germane}, if it possess a Levi factor $\liel$ which is the complexification of a real $\theta$-stable subalgebra $\liel_0\subseteq\lieg$. This is equivalent to saying that $\liel$ is the complexified Lie algebra of a $\theta$-stable closed subgroup of $G$. A canonical choice for $L$ is given by the intersection of the normalizers of $\lieq$ and $\theta(\lieq)$ in $G$. Then $(\liel,L\cap K)$ is again a reductive pair, corresponding to the reductive Lie group $L$, whose Lie algebra is $\liel$, and may be thought of as the Levi factor, that we call the {\em Levi pair} of the {\em parabolic pair} $(\lieq,L\cap K)$.

In the sequel we always assume our Levi factors and Cartan subalgebras to be $\theta$-stable, and parabolic subalgebras to be germane. If $\lieq=\liel+\lieu$ is a Levi-decomposition, we always assume $\liel$ to be $\theta$-stable. Then $L$ and a fortiori $L\cap K$ normalize $\lieu$, $\lieu^-$, $\theta(\lieu)$, $\theta(\lieu^-)$.

Take for example the reductive pair $(\liegl_n,\Oo(n))$. Then the subalgebra $\lieb\subseteq\liegl_n$ consisting of all upper triangular matrices is a germane parabolic subalgebra, because it has a Levi factor consisting of all diagonal matrices and those are evidently defined over $\RR$. In this case it is even true that all of $\lieb$ is defined over $\RR$, as it is the complexification of the algebra of real upper triangular matrices. Therefore $L=\GL_1(\RR)^n$ here, and $L\cap\Oo(n)=\{\pm1\}^n$.

There is another important case. Start with a maximal torus $T\subseteq\Oo(n)$, say with Lie algebra consisting of matrices of the shape
$$
\begin{pmatrix}
 0&t_1&  & \\
-t_1&0&  & \\
  & & 0&t_2 & \\
  & &-t_2&0 & \\
  & &  &  &\ddots\\
\end{pmatrix}.
$$
Then we may choose a complementary maximal abelian subspace in the space $\liep$ of symmetric matrices. Explictly we choose the matrices of the form
$$
\begin{pmatrix}
a_1& 0&  & \\
0  &a_1&  & \\
  & & a_2&0 & \\
  & & 0&a_2 & \\
  & &  &  &\ddots\\
\end{pmatrix},
$$
where we have a single entry $a_{\frac{n+1}{2}}$ at position $n\times n$ if $n$ is odd. Then together this gives us a Cartan subalgebra $\liel$ of $\liegl_n$ whose corresponding Lie group $L$ is isomorphic to $\GL_1(\CC)^{\frac{n-\delta}{2}}\times\GL_1(\RR)^\delta$, where $\delta=0$ if $n$ is even and $\delta=1$ if $n$ is odd. Therefore $L\cap\Oo(n)\cong U(1)^{\frac{n-1}{2}}\times\{\pm1\}^\delta$.

Then there is a germane parabolic subalgebra $\lieq\subseteq\liegl_n$ with Levi factor $\liel$. However $\lieq$ is not defined over $\RR$. But contrary to $\lieb$ above it turns out to be $\theta$-stable. $\lieq$ is of particular importance, because it contains a Borel subalgebra of $\lieso_n$, and therefore turns out to be very useful in the study of restrictions from $\GL_n(\RR)$ to $\Oo(n)$.

We see that although a germane parabolic subalgebra $\lieq$ has a Levi factor which corresponds to a subgroup $L\subseteq G$, the nilpotent radical $\lien$ of $\lieq$ need not be defined over $\RR$, i.e.\ in general there is no subgroup $N\subseteq G$ giving rise to $\lien$. Conceptually there are two extreme cases. We say that $\lieq$ is {\em real} if $\lieq$ is the complexification of $\lieq_0:=\lieq\cap\lieg_0$. We say that $\lieq$ is {\em $\theta$-stable} if $\theta(\lieq)=\lieq$.

\begin{proposition}\label{prop:parabolic}
Let $\lieh\subseteq\lieg$ be $\theta$-stable Cartan subalgebra. The we find a germane Borel subalgebra $\lieq$ containing $\lieh$ with the property that there is a $\theta$-stable parabolic $\lieq'\subseteq\lieg$ with Levi decomposition $\lieq'=\liel'+\lieu'$ and a real parabolic subalgebra $\lieq''\subseteq\liel'$ with Levi decomposition $\lieq''=\lieh+\lieu''$ such that
$$
\lieq=\lieh+(\lieu''+\lieu')
$$
is a Levi decomposition of $\lieq$.
\end{proposition}

\begin{proof}
We write $\lieh_0=\liet_0+\liea_0$ where $\liet_0=\lieh_0\cap\liek$ and $\liea_0$ is the vector part respectively. Now for any finite number of elements $h_1,\dots,h_r\in i\liet_0$ all eigenvalues of $\ad(h_i)$ are real and we define $\lieu_{(h_1,\dots,h_r)}$ as the sum of the simultaneous eigenspaces of $\ad(h_i)$ in $\lieg$ for simultaneous positive eigenvalues, where positivity is defined lexicographically. Then we find $h_1,\dots,h_r$ which satisfy the following maximality condition:
\begin{itemize}
\item[(m)]
$$
\dim\lieu_{(h_1,\dots,h_r)}
$$
is maximal among all choices of elements $h_1,\dots,h_r$.
\end{itemize}
Then (m) is equivalent to
$$
\liel'\;:=\;Z_\lieg(\CC h_1+\cdots+\CC h_r)\;=\;\bigcap_{i=1}^r\ker\ad(h_i).
$$
This guarantees that with $\lieu':=\lieu_{(h_1,\dots,h_r)}$ we obtain a $\theta$-stable germane parabolic subalgebra $\lieq':=\liel'+\lieu'$.

Now choose a Borel subalgebra $\lieq''\subseteq\liel'$ with Levi decomposition $\lieq''=\lieh+\lieu''$. We claim that $\lieq''$ is real. This is the same to say that for any $h'\in \liet_0+\liea_0$, its adjoint action on $\lieu''$ has only real eigenvalues. So assume that this is not the case.

We have the decomposition $h'=it_0+a_0$ with $t_0\in i\liet_0$ and $a_0\in\liea_0$. By our assumption $t_0\neq 0$ and $\ad(t_0)$ has a non-zero real eigenvalue as an endomorphism of $\lieu''$, that we may assume to be positive. We write $E_0\subseteq\lieu''$ for the corresponding eigenspace. As $E_0$ lies in the kernel of all $\ad(h_i)$ we have
$$
E_0\;\subseteq\;\lieu_{(h_1,\dots,h_r,t_0)}.
$$
Therefore $\lieu_{(h_1,\dots,h_r,t_0)}$ strictly contains $\lieu'$. This is a contradiction to (m), concluding the proof.
\end{proof}

Proposition \ref{prop:parabolic} tells us that all germane Cartan subalgebras are Levi factors of Borel subalgebras that are constructed out of real and $\theta$-stable ones. As this class is of particular importance to us, we introduce the following terminology.

Let $\lieq\subseteq\lieg$ be any germane parabolic subalgebra. We say that $\lieq$ is {\em constructible} if there is a chain of germane parabolic subalgebras
$$
\lieq=\lieq_r\;\subseteq\;\lieq_{r-1}\;\subseteq\;\cdots\;
\subseteq\;\lieq_1\;\subseteq\;\lieq_0=\lieg
$$
with Levi decompositions
$$
\lieq_i\;=\;\liel_i+\lieu_i
$$
with Levi pairs $(\liel_i,L_i\cap K)$ such that for all $0\leq i<r$
$$
\liel_i\cap \lieq_{i+1}\;\subseteq\;\liel_i
$$
is a real or $\theta$-stable parabolic in $(\liel_i,L_i\cap K)$.

We say that the Levi pair $(\liel,L\cap K)$ of a constructible parabolic pair $(\lieq,L\cap K)$ is {\em constructible}.

An important corollary of Proposition \ref{prop:parabolic} is

\begin{corollary}\label{cor:constructible}
Ever $\theta$-stable Cartan subpair $(\lieh,H\cap K)\subseteq(\lieg,K)$ is constructible, in particular the reductibe Lie group $G$ corresponding to $(\lieg,K)$ is covered by the conjugates of reductive subgroups $L\subseteq G$ whose corresponding reductive pairs $(\liel,L\cap K)$ are constructible. The $L$ may be chosen the be associated to Cartan pairs.
\end{corollary}

\subsection{Lie algebra cohomology}

Fix a $(\lieg,K)$-module $X$, and consider a parabolic subpair $(\lieq,L\cap K)$ of $(\lieg,K)$ with Levi decomposition $\lieq=\liel+\lieu$. Consider the complex
$$
C_{\lieq}^q(X)\;:=\;\Hom_\CC(\bigwedge^q\lieu,X),
$$
with differential
$$
d:\;C_{\lieq}^q(X)\to C_{\lieq}^{q+1}(X),
$$
$$
f\;\mapsto\;df
$$
where
$$
[df](u_0\wedge u_1\wedge\cdots\wedge u_q)\;:=\;
\sum_{i=0}^q(-1)^i u_i\cdot f(u_0\wedge\cdots\wedge \widehat{u}_i\wedge\cdots\wedge u_q)
$$
$$
\;+\;
\sum_{i<j}^q(-1)^{i+j} f([u_i,u_j]\wedge u_0\wedge\cdots\wedge \widehat{u}_i\wedge\cdots\wedge\widehat{u}_j\wedge\cdots\wedge u_q).
$$
It is easily seen that $d\circ d=0$, hence $(C_{\lieq}^q(X),d)$ is a chain complex. We write $Z^q(X)$ for the kernel of $d:C_{\lieq}^q(X)\to C_{\lieq}^{q+1}(X)$ and $Z^q(X)$ for the image $d(C_{\lieq}^{q-1}(X))$. The reductive pair $(\liel,L\cap K)$ acts on $C_{\lieq}^q(X)$ and $d$ is equivariant for this action. As a result the cohomology
$$
H^q(\lieu; X)\;:=\;Z^q(X)/B^q(X)
$$
of this complex comes with a natural $(\liel,L\cap K)$-module structure.

By construction $C_{\lieq}^q(X)=0$ for all $q<0$ and $q>\dim\lieu$, which implies that $H^q(\lieu;X)=0$ in those degrees. The map \eqref{eq:tensorhom} provides us with an isomorphism
\begin{equation}
C_{\lieq}^q(X)\;\cong\;\bigwedge^q\lieu^*\otimes_\CC X
\label{eq:tensorcomplex}
\end{equation}
of $(\liel,L\cap K)$-modules.

\subsection{Homology and Poincar\'e duality}

Similarly we have a homology theory that is constructed dually via the complex
$$
C_q(X)\;:=\;\bigwedge^q\lieu\otimes_\CC X,
$$
with differential
$$
d:\;C_{q+1}(X)\to C_{q}(X),
$$
$$
d(u_0\wedge u_1\wedge\cdots\wedge u_q\otimes x)\;:=\;
\sum_{i=0}^q(-1)^{i+1} (u_0\wedge\cdots\wedge \widehat{u}_i\wedge\cdots\wedge u_q)\otimes(u_i\cdot x)
$$
$$
\;+\;
\sum_{i<j}^q(-1)^{i+j} [u_i,u_j]\wedge u_0\wedge\cdots\wedge \widehat{u}_i\wedge\cdots\wedge\widehat{u}_j\wedge\cdots\wedge u_q\otimes x.
$$
Note the sign change in the first sum. We denote the cycles in degree $q$ as $Z_q(X)$ and the boundaries with $B_q(X)$. Again the differential is $(\liel,L\cap K)$-linear and the homology
$$
H_q(\lieu; X)\;:=\;Z_q(X)/B_q(X)
$$
is an $(\liel,L\cap K)$-module.

\begin{proposition}[Easy duality]\label{prop:easyduality}
For any $(\lieq,L\cap K)$-module and any degree $q$ we have a natural canonical isomorphism
$$
H^q(\lieu; X^\vee)\;\cong\;
H_q(\lieu; X)^\vee.
$$
of $(\liel,L\cap K)$-modules. Here the supscript $\cdot^\vee$ denotes on both sides the $L\cap K$-finite dual.
\end{proposition}

This duality is sometimes called easy duality, as opposed to hard duality (Poincar\'e duality) that we discuss below.

\begin{proof}
The natural perfect $(\lieq,L\cap K)$-equivariant pairing
$$
\langle\cdot,\cdot\rangle:\;\;\;X^\vee\times X\;\to\;\CC,
$$
$$
(y,x)\;\mapsto\;y(x)
$$
induces a natural perfect pairing
$$
\langle\cdot,\cdot\rangle:\;\;\;C_{\lieq}^q(X^\vee)\times C_q(X)\;\to\;\CC,
$$
$$
(f,u_0\wedge\cdots\wedge u_q\wedge x)\;\mapsto\;\langle f(u_0\wedge\cdots\wedge u_q),x\rangle,
$$
which is $(\liel,L\cap K)$-equivariant. We claim that this pairing descends to a perfect pairing
\begin{equation}
\langle\cdot,\cdot\rangle:\;\;\;H^q(\lieu;X^\vee)\times H_q(\lieu;X)\;\to\;\CC.
\label{eq:cohomologyhomologypairing}
\end{equation}
To see that this is well defined, we first observe that the maps
$$
d:\;\;\;C_{\lieq}^{q}(\lieu; X^\vee)\to
C_{\lieq}^{q+1}(\lieu; X^\vee),
$$
and
$$
d:\;\;\;C_{q+1}(\lieu; X)\to
C_{q}(\lieu; X),
$$
are adjoint with respect to $\langle\cdot,\cdot\rangle$. Indeed, we have
$$
\langle df,u_0\wedge\cdots u_q\otimes x\rangle\;=\;
\langle[df](u_0\wedge u_1\wedge\cdots\wedge u_q),x\rangle\;=\;
$$
$$
\sum_{i=0}^q(-1)^i \langle u_i\cdot f(u_0\wedge\cdots\wedge \widehat{u}_i\wedge\cdots\wedge u_q),x\rangle
$$
$$
\;+\;
\sum_{i<j}^q(-1)^{i+j} \langle f([u_i,u_j]\wedge u_0\wedge\cdots\wedge \widehat{u}_i\wedge\cdots\wedge\widehat{u}_j\wedge\cdots\wedge u_q),x\rangle\;=\;
$$
$$
\sum_{i=0}^q(-1)^{i+1} \langle f(u_0\wedge\cdots\wedge \widehat{u}_i\wedge\cdots\wedge u_q),u_i\cdot x\rangle
$$
$$
\;+\;
\sum_{i<j}^q(-1)^{i+j} \langle f([u_i,u_j]\wedge u_0\wedge\cdots\wedge \widehat{u}_i\wedge\cdots\wedge\widehat{u}_j\wedge\cdots\wedge u_q),x\rangle\;=\;
$$
$$
\langle f, d(u_0\wedge u_1\wedge\cdots\wedge u_q\otimes x)\rangle.
$$
Therefore, if $f\in C_{\lieq}^{q-1}(X^\vee)$ and $g\in Z_q(X)$ is a cycle, i.e.\ lies in the kernel of $d$,
$$
\langle df,g\rangle \;=\;\langle f,dg\rangle\;=\;\langle f,0\rangle\;=\;0,
$$
and similarly for a cocycle $f\in Z^q(X^\vee)$ and $g\in C_{q-1}(X)$
$$
\langle f,dg\rangle \;=\;\langle df,g\rangle\;=\;\langle 0,g\rangle\;=\;0.
$$
Therefore the pairing \eqref{eq:cohomologyhomologypairing} is well defined. To see that it is non-degenerate, we again use the same adjointness relation: Assume that $f\in Z^q(X^\vee)$ is a cocycle with
$$
\langle f,g\rangle\;=\;0
$$
for all cycles $g\in Z_q(X)$. This means that
$$
f\;\in\;Z^q(X^\vee)\cap Z_q(X)^\bot.
$$
By adjointness we have
$$
Z_q(X)\;=\;B^q(X^\vee)^\bot,
$$
because for any $h\in C_q(X)$ we have the equivalence
$$
\forall e\in C_{\lieq}^{q-1}(X^\vee):\;\langle de,h\rangle=0\;\;\Leftrightarrow\;\;
\forall e\in C_{\lieq}^{q-1}(X^\vee):\;\langle e,dh\rangle=0\;\;\Leftrightarrow\;\;dh=0.
$$
Therefore
$$
Z_q(X)^\bot\;=\;\left(B^q(X^\vee)^\bot\right)^\bot=B^q(X^\vee),
$$
and consequently $f$ is a coboundary. Switching the roles of $f$ and $g$ this completes the proof of non-degeneracy of \eqref{eq:cohomologyhomologypairing}.
\end{proof}

The following statement is a consequence of hard duality, and in this context also referred to as {\em Poincar\'e duality}. The {\em hard} refers to the generalization involving production or induction along the compact group as well, and in this case it is harder to proof. Our setting is not harder than above.

\begin{proposition}[Poincar\'e duality]\label{prop:hardduality}
For any $(\lieq,L\cap K)$-module and any degree $q$ we have a natural canonical isomorphism
$$
H_q(\lieu; X\otimes_\CC\bigwedge^{\dim\lieu}\lieu^*)\;\cong\;
H^{\dim\lieu-q}(\lieu; X).
$$
of $(\liel,L\cap K)$-modules.
\end{proposition}

\begin{proof}
Consider the natural map
$$
\bigwedge^{q}\lieu\otimes_\CC
\bigwedge^{\dim\lieu-q}\lieu\otimes_\CC
\bigwedge^{\dim\lieu}\lieu^*\;\to\;
\CC
$$
given explicitly by
$$
u\otimes
v\otimes
w
\;\mapsto\;w(u\wedge v).
$$
It is $(\liel,L\cap K)$-equivariant, and every non-zero $w$ identifies $\bigwedge^{\dim\lieu-q}\lieu$ as the dual of $\bigwedge^{q}\lieu$, i.e.\ it gives rise to an isomorphism
$$
\eta_w:\;\;\;\bigwedge^{\dim\lieu-q}\lieu\;\to\;(\bigwedge^{q}\lieu)^*.
$$
This induces an $(\liel,L\cap K)$-isomorphism
$$
\bigwedge^{q}\lieu\otimes_\CC
X\otimes_\CC
\bigwedge^{\dim\lieu}\lieu^*\;\to\;
\left(\bigwedge^{\dim\lieu-q}\lieu\right)^*\otimes_\CC
X,
$$
$$
u\otimes x\otimes w\;\mapsto\;\eta(w)\otimes x.
$$
The right hand side has a natural identification with the complex computing cohomology, hence we get an isomorphism
$$
\delta:\;\;\;C_q(X\otimes_\CC\bigwedge^{\dim\lieu}\lieu^*)\;\to\;
C_{\lieq}^{\dim\lieu-q}(X).
$$
Let us verify that $\delta$ commutes with the differentials up to sign:
$$
[\delta(d(u_0\wedge\cdots\wedge u_{q}\otimes x\otimes w)](v_q\wedge\cdots\wedge v_{\dim\lieu})\;=\;
$$
$$
\sum_{i=0}^q(-1)^{i+1} \delta(u_0\wedge\cdots\wedge \widehat{u}_i\wedge\cdots\wedge u_{q}\otimes u_i(x\otimes w))(v_q\wedge\cdots\wedge v_{\dim\lieu})\;+\;
$$
$$
\sum_{i<j}^q(-1)^{i+j} \delta([u_i,u_j]\wedge u_0\wedge\cdots\wedge \widehat{u}_i\wedge\cdots\wedge\widehat{u}_j\wedge\cdots\wedge u_q\otimes x\otimes w)(v_q\wedge\cdots\wedge v_{\dim\lieu})
\;=\;
$$
$$
\sum_{i=0}^q(-1)^{i+1} \delta(u_0\wedge\cdots\wedge \widehat{u}_i\wedge\cdots\wedge u_{q}\otimes (u_i\cdot x)\otimes w)(v_q\wedge\cdots\wedge v_{\dim\lieu})
\;+\;
$$
$$
\sum_{i<j}^q(-1)^{i+j} \delta([u_i,u_j]\wedge u_0\wedge\cdots\wedge \widehat{u}_i\wedge\cdots\wedge\widehat{u}_j\wedge\cdots\wedge u_q\otimes x\otimes w)(v_q\wedge\cdots\wedge v_{\dim\lieu}),
$$
because $\lieu$ acts trivially on $w$. We get
$$
\sum_{i=0}^q(-1)^{i+1} [\eta_w(u_0\wedge\cdots\wedge \widehat{u}_i\wedge\cdots\wedge u_{q})(v_q\otimes (u_i\cdot x)](v_q\wedge\cdots\wedge v_{\dim\lieu})\;+\;
$$
$$
\sum_{i<j}^q(-1)^{i+j} [\eta_w([u_i,u_j]\wedge u_0\wedge\cdots\wedge \widehat{u}_i\wedge\cdots\wedge\widehat{u}_j\wedge\cdots\wedge u_q\otimes x)](v_q\wedge\cdots\wedge v_{\dim\lieu})\;=\;
$$
$$
\sum_{i=0}^q(-1)^{i+1} w(u_0\wedge\cdots\wedge \widehat{u}_i\wedge\cdots\wedge u_{q}\wedge v_q\wedge\cdots\wedge v_{\dim\lieu})\cdot (u_i\cdot x)\;+\;
$$
$$
\sum_{i<j}^q(-1)^{i+j} w([u_i,u_j]\wedge u_0\wedge\cdots\wedge \widehat{u}_i\wedge\cdots\wedge\widehat{u}_j\wedge\cdots\wedge u_q\wedge v_q\wedge\cdots\wedge v_{\dim\lieu})\cdot x.
$$
Analogously we get
$$
[d(\delta(u_0\wedge\cdots\wedge u_{q}\otimes x\otimes w))](v_q\wedge\cdots\wedge v_{\dim\lieu})\;=\;
$$
$$
\sum_{j=q}^{\dim\lieu}(-1)^{j-q} v_j\delta(u_0\wedge\cdots\wedge u_{q}\otimes x\otimes w)(v_q\wedge\cdots\wedge \widehat{v}_j\wedge\cdots\wedge v_{\dim\lieu})\;+\;\cdots\;=\;
$$
$$
\sum_{j=q}^{\dim\lieu}(-1)^{j-q} v_j\eta_w(u_0\wedge\cdots\wedge u_{q}\otimes x)(v_q\wedge\cdots\wedge \widehat{v}_j\wedge\cdots\wedge v_{\dim\lieu})\;+\;\cdots\;=\;
$$
$$
\sum_{j=q}^{\dim\lieu}(-1)^{j-q} w(u_0\wedge\cdots\wedge u_{q}\wedge v_q\wedge\cdots\wedge \widehat{v}_j\wedge\cdots\wedge v_{\dim\lieu})\cdot (v_j\cdot x)\;+\;
$$
$$
\sum_{k<l}^q(-1)^{k+l} w([v_k,v_l]\wedge u_0\wedge\cdots\wedge u_q\wedge v_q\wedge\cdots\wedge \widehat{v}_k\wedge\cdots\wedge\widehat{v}_l\wedge\cdots\wedge v_{\dim\lieu})\cdot x,
$$
Summing up we conclude that
$$
\delta(d(u\otimes x\otimes w))(v)\;+\;(-1)^{q}d(\delta(u\otimes x\otimes w))(v)\;=\;
$$
$$
\sum_{i=0}^q(-1)^{i+1} w(u_0\wedge\cdots\wedge \widehat{u}_i\wedge\cdots\wedge u_{q}\wedge v_q\wedge\cdots\wedge v_{\dim\lieu})\cdot (u_i\cdot x)\;+\;
$$
$$
\sum_{j=q}^{\dim\lieu}(-1)^{j} w(u_0\wedge\cdots\wedge u_{q}\wedge v_q\wedge\cdots\wedge \widehat{v}_j\wedge\cdots\wedge v_{\dim\lieu})\cdot (v_j\cdot x).
$$
For any fixed $w$ and $x$ we may interpret this expression as an alternating $X$-valued $(\dim\lieu+1)$-form on $\lieu$, and every such form is zero.
\end{proof}

\begin{corollary}\label{cor:duality}
We have a canonical isomorphism
$$
H^q(\lieu; (X\otimes_\CC\bigwedge^{\dim\lieu}\lieu^*)^\vee)\;\cong\;
H^{\dim\lieu-q}(\lieu; X)^\vee.
$$
of $(\liel,L\cap K)$-modules.
\end{corollary}

We remark that this isomorphism is even an isomorphism of universal $\delta$-functores. In particular it is compatible with the connecting morphisms.

\begin{proof}
From Propositions \ref{prop:easyduality} and \ref{prop:hardduality} we get
$$
H^q(\lieu; (X\otimes_\CC\bigwedge^{\dim\lieu}\lieu^*)^\vee)\;\cong\;
H_q(\lieu; X\otimes_\CC\bigwedge^{\dim\lieu}\lieu^*)^\vee\;\cong\;
H^{\dim\lieu-q}(\lieu; X)^\vee.
$$
\end{proof}

\subsection{The K\"unneth formula}

Consider two $(\lieg,K)$-modules $V$ and $W$. Then the tensor product $V\otimes_\CC W$ is a $(\lieg\times\lieg,K\times K)$-mdule and $\lieq\times\lieq$ is a parabolic subpair of the latter reductive pair. Then the $\lieu\times\lieu$-cohomology of $V\otimes_\CC W$ is calculated via the standard complex
$$
\Hom_\CC(\bigwedge^n\lieu\times\lieu,V\otimes_\CC W),
$$
which decomposes naturally into
$$
\bigoplus_{p+q=n}\Hom_\CC(\bigwedge^p\lieu,V)\otimes_\CC\Hom_\CC(\bigwedge^q\lieu,W),
$$
via the vector space isomorphism
$$
\phi\;:=\sum_{p+q=n}\phi^{p,q},
$$
where
$$
\phi^{p,q}:\;\;\Hom_\CC(\bigwedge^p\lieu,V)\otimes\Hom_\CC(\bigwedge^q\lieu,W)\;\to\;\Hom_\CC(\bigwedge^n\lieu\times\lieu,V\otimes_\CC W),
$$
$$
f\otimes g\;\;\;\mapsto\;\;\; [U_1\wedge U_2\mapsto f(v)\otimes g(w)]
$$
whenever $U_1\in\bigwedge^p\lieu$ and $U_2\in\bigwedge^q\lieu$ and the other terms are mapped to zero. It is not hard to see that this decomposition is compatible with the differentials on the complexes. Hence we deduce
\begin{proposition}[K\"unneth formula]\label{prop:kuenneth}
For any $V,W\in\mathcal C(\lieg,K)$, and any $n$ there is a canonical isomorphism
$$
\bigoplus_{p+q=n}H^p(\lieu;V)\otimes_\CC H^q(\lieu;W)\;\cong\;
H^{n}(\lieu\times\lieu;V\otimes_\CC W),
$$
of $(\liel\times\liel,(L\cap K)\times (L\cap K))$-modules.
\end{proposition}

\subsection{The Hochschild-Serre spectral sequence}

The Hochschild-Serre spectral sequence \cite{hochschildserre1953}, \cite[Chapter 5, Section 10]{book_knappvogan1995} will be our main fundamental tool in our algebraic character theory. It is a technically forbidding subject to show its existence, which has been discussed in detail in many standard text books.

Suppose we are given an inclusion of reductive pairs $(\lieh,N)\to(\lieg,K)$ and germane parabolic subalgebras $\liep\subseteq\lieh$ and $\lieq\subseteq\lieg$, satisfying $\liep=\lieq\cap\lieh$ with Levi decompositions $\liep=\liem+\lien$, $\lieq=\liel+\lieu$ subject to the condition $\liem\subseteq\liel$, i.e.\ $\liem=\liel\cap\lieh$. 
The main idea here to describe the structure of $\lieu$-cohomology as a $(\liem,M\cap N)$-module via the $\lien$-cohomology.
Hochschild and Serre gave a structural recipe, i.e.\ a spectral sequence, which inductively achieves this.

More concretely for any $(\lieq,L\cap K)$-module $X$ they start with the $(\liem,M\cap N)$-modules
\begin{equation}
E_1^{p,q}\;:=\;\bigwedge^p(\lieu/\lien)^*\otimes_\CC H^q(\lien; X).
\label{eq:e1term}
\end{equation}
Then they inductively construct a collection $(E^{p,q}_r)_{p,q\in\ZZ,r\geq 1}$ of $(\liem,M\cap N)$-modules together with morphisms
$$
d^{p,q}_r:\;\;\;E^{p,q}_r\;\to\;E^{p+r,q-r+1}_r,
$$
subject to the condition
\begin{equation}
d_r^{p+r,q-r+1}\circ d_r^{p,q}\;=\;0,
\label{eq:dschicht}
\end{equation}
and isomorphisms
\begin{equation}
\alpha_r^{p,q}:\;\;\;H^{p,q}(E_r^{\bullet,\bullet})\;\to\;E^{p,q}_{r+1},
\label{eq:alphaiso}
\end{equation}
where the bigraded cohomology is defined via the above differential as
\begin{equation}
H^{p,q}(E_r^{\bullet,\bullet})\;\;:=\;\;\kernel d^{p,q}_r/\image d_r^{p-r,q+r-1}.
\label{eq:Hschicht}
\end{equation} 
In particular this data builds an explicit bridge between the $r$-th layer $E_r^{p,q}$ and the $(r+1)$-th one. They showed that for $r_0$ large enough the differentials of the $r_0$-th layer and above vanish. Therefore \eqref{eq:Hschicht} becomes
$$
H^{p,q}(E_{r_0}^{\bullet,\bullet})\;\;=\;\;
E^{p,q}_{r_0}\;\;=:\;\;
E^{p,q}_\infty.
$$
The same is true for any $r\geq r_0$, i.e.\ we may naturally identify $E^{p,q}_r$ with $E^{p,q}_\infty$ via the isomorphisms \eqref{eq:alphaiso}.

Finally Hochschild and Serre construct a decreasing filtration of
$$
E^n\;:=\;H^n(\lien; X)
$$
via submodules
$$
0\;\subseteq\cdots\subseteq\;F^{p+1}E^{n}\;\subseteq\;F^pE^{n}\;\subseteq\;
F^{p-1}E^{n}\;\subseteq\;\cdots\;\subseteq\;F^0E^{n}\;=\;E^n,
$$
satisfying for $p$ large enough
$$
F^pE^n\;=\;0
$$
for all $n$, together with isomorphisms
$$
\beta^{p,q}:\;\;\;E_\infty^{p,q}\;\to\;F^pE^{p+q}/F^{p+1}E^{p+q}
$$
for all $p$ and $q$.

In abstract terms this is equivalent to saying that they constructed a spectral sequence $(E_r^{p,q},E^n)_{r,p,q,n}$ with differential of bidegree $(r,r-1)$ and $E_1$-term given by \eqref{eq:e1term}, converging to $H^{p+q}(\lieu; X)$. This is convergens expressed by the notation
$$
E_1^{p,q}\;\;\;\Longrightarrow\;\;\; H^{p+q}(\lieu; X).
$$
Their construction is functorial in $X$ and all maps involved are $(\liem,M\cap N)$-equivariant.

If there is another parabolic subalgbra $\lieq'\subseteq\lieg$ containing $\lieq$ with Levi factor $\lieh$, and nilpotent radical $\lieu'$ say, then the second layer is given by
$$
E_2^{p,q}\;=\;H^p(\lien; H^q(\lieu'; X)).
$$
In general there is no such $\lieq'$, and then the Hochschild-Serre spectral sequence is nothing but the Grothendieck spectral sequence for the composition of functor
$$
H^0(\lien;\cdot)\circ H^0(\lieu'; X)\;=\;
H^0(\lieu;\cdot).
$$

A spectral sequence contains a lot of information, which is often not easy to extract. However Euler characteristics are invariant in spectral sequences, i.e.\ the Euler characteristic of the $E_r$-terms agree for all $r$, and in particular they are the same as the Euler characteristic of the $E^n$-terms.

\subsection{Properties preserved by cohomology}

\begin{theorem}\label{thm:inheritance}
For any parabolic subpair $(\lieq,L\cap K)$ of a reductive pair $(\lieg,K)$, and any $q\in\ZZ$ the functor
$$
H^q(\lieu;-):\;\;\;{\mathcal C}(\lieg,K)\to{\mathcal C}(\liel,L\cap K)
$$
preserves the following properties:
\begin{itemize}
\item[(i)] finite-dimensionality,
\item[(ii)] $Z(\lieg)$- resp.\ $Z(\liel)$-finiteness,
\item[(iii)] admissibility for $\theta$-stable $\lieq$,
\item[(iv)] finite length for constructible $\lieq$,
\item[(v)] discrete decomposability for constructible $\lieq$,
\item[(vi)] discrete decomposability with finite multiplicities for constructible $\lieq$.
\end{itemize}
\end{theorem}

\begin{proof}
The assertions (i) and (iii) follow easily. If $X$ is finite-dimensional or finitely generated, then all $C_{\lieq}^q(X)$ are finite-dimensional resp.\ finitely generated too. Therefore $H^q(\lieu;X)$ is finite-dimensional in the first case, and finitely generated as a $(\liel,L\cap K)$-module because $U(\liel)$ is noetherian.

The hard part is proving (ii), for which we refer to the literature \cite{casselmanosborne1975}, \cite[Theorem 7.56]{book_knappvogan1995}. One usually shows a stronger assertion: If $X$ is $Z(\lieg)$-finite, then $H^q(\lieu;X)$ is $Z(\liel)$-finite and if the latter has a non-trivial $\chi_\lambda$-primary component for some character $\lambda\in\lieh^*$, $\lieh\subseteq\liel$ a $\theta$-stable Cartan, then there is a $w\in W(\lieg,\lieh)$ such that the $\chi_{w(\lambda+\rho(\lien))}$-primary component of $X$ is nonzero, where $\lien$ is the nilpotent radical of a Borel $\lieb$ containing $\lieh$. We will need this below in our proof of (vi).

The proof of (iii) is an application of the Hochschild-Serre spectral sequence that we discuss below. The idea is to reduce the problem via this spectral sequence to the cohomology of the underlying compact pairs, where it becomes obvious thanks to Kostant's Theorem (Theorem \ref{thm:kostant} below). For details we refer to \cite[Corollary 5.140]{book_knappvogan1995}.

Assertion (iv) follows in the $\theta$-stable case from (ii) and (iii).  The real case is treated in \cite[Section 2]{hechtschmid1983}, and the general case follows from these two cases by Proposition \ref{prop:parabolic} via the Hochschild-Serre spectral sequence.

To proove (v), let
$$
X=\varinjlim X_i
$$
be a discretely decomposable $(\lieg,K)$-module, i.e.\ we assume the $X_i$ to be of finite length. Then by (v) the $\lieu$-cohomology of the $X_i$ is of finite length too, as is by Poincar\'e duality the homology of
$$
X_i\otimes_\CC\bigwedge^{\dim\lieu}\lieu^*.
$$
Now homology and the tensor product commute with direct limits. Therefore
$$
H_{\dim\lieu-q}(\lieu;X\otimes_\CC\bigwedge^{\dim\lieu}\lieu^*)\;=\;
\varinjlim H_{\dim\lieu-q}(\lieu;X_i\otimes_\CC\bigwedge^{\dim\lieu}\lieu^*)
$$
By Poincar\'e duality the left hand side computes the $q$-th $\lieu$-cohomology of $X$, and the right hand side the corresponding cohomology of $X_i$. This proves (v).

To proof (vi), we may assume without loss of generality that $K$ is connected. Now our observation in the proof of (ii) above allows us together with (v) to conclude the validity of (vi) thanks to Harish-Chandra's Theorem \ref{thm:harishchandrafinite}.
\end{proof}

\begin{theorem}[Kostant, Bott]\label{thm:kostant}
If $K$ is connected and the Levi factor $\liel$ of $\lieq=\liel+\lieu$ is a Cartan subalgebra, and if $V$ is an irreducible finite-dimensional representation of $(\lieg,K)$ of $\lieu$-hightest weight $\lambda\in\liel^*$, then for any $0\leq q\leq\dim\lieu$:
$$
H^q(\lieu;V)\;\cong\;
\bigoplus\limits_{\substack{w\in W(\lieg,\liel)\\\ell(w)=q}}
\CC_{w(\lambda+\rho(\lieu))-\rho(\lieu)}
$$
as an $\liel$-module.
\end{theorem}

This Theorem was already implicit in Bott \cite{bott1957}, but became explicit only in \cite{kostant1961}.

\begin{corollary}\label{cor:denominator}
In the above setting we have
$$
H^q(\lieu;\CC)\;\cong\;
\bigoplus\limits_{\substack{w\in W(\lieg,\liel)\\\ell(w)=q}}
\CC_{w(\rho(\lieu))-\rho(\lieu)}.
$$
\end{corollary}

\subsection{Cohomology and Grothendieck groups}

It turns out that $H^q(\lieu;-)$ is the $q$-th right derived functor of $H^0(\lieu;-)$. Therefore it maps short exact sequences
$$
\begin{CD}
0@>>> X@>>> Y@>>> Z@>>> 0
\end{CD}
$$
to long exact sequences
$$
\begin{CD}
0@>>> H^0(\lieu;X)@>>> H^0(\lieu;Y)@>>> H^0(\lieu;Z)
\end{CD}
$$
$$
\begin{CD}
@>>>
H^1(\lieu;X)@>>> H^1(\lieu;Y)@>>> H^1(\lieu;Z)@>>>\cdots
\end{CD}
$$
$$
\begin{CD}
@>>>
H^{\dim\lieu}(\lieu;X)@>>> H^{\dim\lieu}(\lieu;Y)@>>> H^{\dim\lieu}(\lieu;Z)@>>>0.
\end{CD}
$$
A fundamental consequence of this fact and Theorem \ref{thm:inheritance} is that the map
$$
H_\lieq(-):\;\;\;K_?(\lieg,K)\to K_?(\liel,L\cap K)
$$
$$
[X]\;\;\mapsto\;\sum_{q=0}^{\dim\lieu}(-1)^q[H^q(\lieu; X)]
$$
is a well defined group homomorphism for $?\in\{\rm a,\rm df,\rm fl,\rm fd\}$, where in the case $?=\rm a$ we always implicitly assume $\lieq$ to be $\theta$-stable.

\section{Algebraic Characters}

After all the preparation in the previous section, this section contains the main body of our theory.

\subsection{Definition}

Fix a reductive pair $(\lieg,K)$ and a germane parabolic subalgebra $\lieq\subseteq\lieg$ with Levi decomposition $\lieq=\liel+\lieu$. Define the (relative) {\em Weyl element} (with respect to $\lieq$) as
$$
W_{\lieq}\;:=\;H_\lieq({\bf1})\;\in\;K_{\rm fd}(\liel, L\cap K).
$$
The isomorphism \eqref{eq:tensorcomplex} tells us that
\begin{equation}
W_{\lieq}\;=\;\sum_{q=0}^{\dim\lieu}(-1)^q[\bigwedge^q\lieu^*].
\label{eq:wformula}
\end{equation}
Choose any $?\in\{\rm a,\rm df,\rm fl,\rm fd\}$. From now on we assume that $\lieq$ is $\theta$-stable if $?=\rm a$ and constructible in the other cases.

Then Proposition \ref{prop:tensorstable} tells us that the category $\mathcal C_{?}(\liel,L\cap K)$ is stable under tensoring with finite-dimensional modules, and in particular the localization
$$
C_{\lieq,?}(\liel,L\cap K)\;:=\;K_?(\liel,L\cap K)[W_\lieq^{-1}]
$$
is well defined. As we saw before, Theorem \ref{thm:inheritance} and the long exact sequence for cohomology, tells us that $\lieu$-cohomology defines a map map $H_\lieq$ between the $K$-groups of $(\lieg,K)$ and $(\liel,L\cap K)$, and we set
$$
c_\lieq:\;\;\;K_{?}(\lieg,K)\;\to\;C_{\lieq,?}(\liel,L\cap K),
$$
$$
X\;\mapsto\;\frac{H_\lieq(X)}{W_\lieq}.
$$

\subsection{Formal properties}

The fundamental properties of $c_\lieq$ are that it satisfies the analogues of axioms (A) and (M). However as in general multiplication is only partially defined, as for example the tensor product of finite length modules need not be of finite length again.

\begin{proposition}[Additivity and Multiplicativity]\label{prop:addmult}
The map $c_\lieq$ is additive and multiplicative in the sense that
$$
c_\lieq(X+Y)\;=\;
c_\lieq(X)+c_\lieq(Y),
$$
$$
c_\lieq({\bf1})\;=\;{\bf1},
$$
and if $X,Y,X\otimes_\CC Y\in\mathcal C_?(\lieg,K)$, then
$$
c_\lieq(X\cdot Y)\;=\;
c_\lieq(X)\cdot c_\lieq(Y)
$$
under the assumption that $\lieq$ is a Borel algebra.
\end{proposition}

All identities are understood in $C_{\lieq,?}(\liel,L\cap K)$.

Our requirement of minimality of $\lieq$ for multiplicativity becomes evident in the proof: In general we cannot garantee that
$$
H^p(\lieu;X)\otimes_\CC H^q(\lieu;Y)\in \mathcal C_?(\liel,L\cap K),
$$
even though we conjecture that this is always the case for $X\otimes_\CC Y\in \mathcal C_?(\lieg,K)$.

\begin{proof}
Additivity is a formal consequence of the long exact sequence for cohomology (without which $c_\lieq$ were not well defined on the Grothendieck group).

As for the multiplicativity, we consider the $\lieu\times\lieu$-cohomology of $X\otimes_\CC Y$ as an $(\liel\times\liel,(L\cap K)\times(L\cap K))$-module. Then the K\"unneth formula from Proposition \ref{prop:kuenneth} tells us that after restricting to the diagonal
$$
\Delta:\;\;\;(\liel,(L\cap K))\;\to\;
(\liel\times\liel,(L\cap K)\times(L\cap K)),
$$
we have an identity
\begin{equation}
H_\lieq(X)\cdot H_\lieq(Y)\;=\;
H_{\lieq\times\lieq}(X\otimes_\CC Y)
\label{eq:hkuenneth}
\end{equation}
in $K_?(\liel,L\cap K)$. Now the stability of Euler characteristics in the Hochschild-Serre spectral sequence for the embedding
$$
\Delta:\;\;\;(\lieg,K)\;\to\;(\lieg\times\lieg,K\times K)
$$
and the parabolic subalgebras $\lieq$ resp.\ $\lieq\times\lieq$ gives
$$
H_{\lieq\times\lieq}(X\otimes_\CC Y)\;=\;
\sum_{p,q}(-1)^{p+q}
[\bigwedge^p(\lieu\times\lieu/\Delta(\lieu))^*]
\cdot[H^q(\Delta(\lieu);X\otimes_\CC Y)]
$$
in $K_?(\liel,L\cap K)$. Now we have an isomorphism
$$
\bigwedge^p(\lieu\times\lieu/\Delta(\lieu))^*\;\cong\;\bigwedge^p\lieu^*
$$
of $(\liel,L\cap K)$-modules, and hence we get
$$
H_{\lieq\times\lieq}(X\otimes_\CC Y)\;=\;W_\lieq\cdot H_\lieq(X\otimes_\CC Y).
$$
Together with \eqref{eq:hkuenneth} this proves the claim.
\end{proof}

An important consequence of the multiplicativity is
\begin{corollary}\label{cor:addloc}
For any $W'\in K_{\rm fd}(\lieg,K)$ the map $c_\lieq$ induces a well defined additive and multiplicative map
$$
c_\lieq:\;\;\;K_{?}(\lieg,K)[W'^{-1}]\;\to\;C_{\lieq,?}(\liel,L\cap K)[W'|_{\liel,L\cap K}^{-1}].
$$
\end{corollary}

\begin{proposition}[Compatibility with restriction]\label{prop:restriction}
The map $c_\lieq$ is compatible with restriction, i.e.\ if $X\in\mathcal C_?(\lieg,K)$ and $X|_{\liel,L\cap K}\in\mathcal C_?(\liel,L\cap K)$ then
$$
c_\lieq(X)\;=\;[X|_{\liel,L\cap K}].
$$
\end{proposition}

This identity again is understood in $C_{\lieq,?}(\liel,L\cap K)$.

\begin{proof}
Suppose that the restriction of $X$ to the Levi pair lies in $\mathcal C_?(\liel,L\cap K)$. Then this is true for the complex $C_{\lieq}^q(X)$ computing cohomology. Now we have the formal identity
$$
H_\lieq(X)\;=\;\sum_{q=0}^{\dim\lieu}(-1)^q[C_{\lieq}^q(X)].
$$
By the formula \eqref{eq:tensorcomplex} we get
$$
\sum_{q=0}^{\dim\lieu}(-1)^q[C_{\lieq}^q(X)]\;=\;
\sum_{q=0}^{\dim\lieu}(-1)^q[\bigwedge^q\lieu^*\otimes_\CC X)]\;=\;
[X]\cdot\sum_{q=0}^{\dim\lieu}(-1)^q[\bigwedge^q\lieu^*].
$$
With the identity \eqref{eq:wformula} we conclude that
\begin{equation}
H_\lieq(X)\;=\;[X]\cdot W_\lieq,
\label{eq:riemannroch}
\end{equation}
is an identity in $K_{?}(\liel,L\cap K)$, which concludes the proof.
\end{proof}

The identity \eqref{eq:riemannroch} is sometimes also referred to as Riemann-Roch formula, for it gives an explicit expression for the Euler-Poincar\'e characteristic of the cohomology.

\subsection{Applications to finite-dimensional modules}

Let us assume in this section for simplicity that $K$ is connected, although all statements remain true with disconnected $K$, although the arguments get a bit more involved.

What we have done so far enables us already to study interesting consequences of our results in the case of finite-dimensional modules.

Forgetting along the inclusion of the Levi pair induces an additive and multiplicative map
$$
{\mathcal F}:\;\;\;K_{\rm fd}(\lieg,K)\;\to\;K_{\rm fd}(\liel,L\cap K).
$$
As multiplication is always defined in the finite-dimensional setting, this map is actually a ring homomorphism.

Consider the following diagram
$$
\begin{CD}
K_{\rm fd}(\lieg,K)@>{\mathcal F}>>K_{\rm fd}(\liel,L\cap K)\\
@| @VVV\\
K_{\rm fd}(\lieg,K)@>{c_\lieq}>>C_{\lieq,\rm fd}(\liel,L\cap K)
\end{CD}
$$
of rings, which is commutative by Proposition \ref{prop:restriction}. By Proposition \ref{prop:connecteddomain} the vertical localization map is a monomorphism. As we already know that $\mathcal F$ is a monorphism, this shows
\begin{theorem}\label{thm:fdmono}
For connected $K$ the algebraic character map
$$
c_\lieq:\;\;\;K_{\rm fd}(\lieg,K)\;\to\;C_{\lieq,\rm fd}(\liel,L\cap K)
$$
is a monomorphism of rings, and it satisfies
$$
c_\lieq(X)\;=\;[X|_{\liel,L\cap K}].
$$
\end{theorem}

Suppose that $\liep$ is a parabolic subpair of $(\lieg,K)$ contained in $\lieq$, with Levi decomposition $\liep=\liem+\lien$. Then the obvious relation
$$
X|_{\liem,M\cap K}\;=\;(X|_{\liel,L\cap K})|_{\liem,M\cap L}
$$
translates to the identity
\begin{equation}
c_{\liep}\;=\;c_{\liep\cap\liel}\circ c_{\lieq},
\label{eq:fdres}
\end{equation}
where $c_{\liep\cap\liel}$ is interpreted as a map
$$
C_{\lieq,\rm fd}(\liel,L\cap K)\;\to\;C_{\liep,\rm fd}(\liem,M\cap K),
$$
which is meaningful thanks to Corollary \ref{cor:addloc}.

The identity \ref{eq:fdres} enables us to study branching problems of finite-dimensional modules via algebraic characters. A remarkable fact is that this identity remains true in general, not only the finite-dimensional case, as we will see shortly.

Another identity is related to duality. As taking duals is an exact functor, it induces an additive map
$$
\cdot^\vee:\;\;\;K_{\rm fd}(\lieg,K)\;\to\;K_{\rm fd}(\lieg,K).
$$
Now for any $X,Y\in\mathcal C_{\rm fd}(\lieg,K)$ we have a (non-canonical) isomorphism
$$
X^\vee\otimes_\CC Y^\vee\;\cong\;(X\otimes_\CC Y)^\vee,
$$
which means that dualization is multiplicative and hence a ring isomorphism on Grothendieck groups. The same is true for the Levi pair, and due to its multiplicativity dualization induces a well defined ring automorphism
$$
\cdot^\vee:\;\;\;C_{\lieq,\rm fd}(\liel,L\cap K)\;\to\;C_{\lieq,\rm fd}(\liel,L\cap K).
$$
As taking duals commutes obviously with restrictions, we get the identity
\begin{equation}
c_{\liep}(X^\vee)\;=\;c_{\lieq}(X)^\vee.
\label{eq:fddual}
\end{equation}
We will see that this identity also generalizes.

One may wonder why we pass to the localization in this setting at all, restriction seems to be good enough. The reason being that even here localization makes explicit calculations possible, as it allows explicit character formulae as in the classical

\begin{theorem}[Weyl character formula]
If $K$ is connected and the Levi factor $\liel$ of $\lieq=\liel+\lieu$ is a Cartan subalgebra, and if $V$ is an irreducible finite-dimensional representation of $(\lieg,K)$ of $\lieu$-hightest weight $\lambda\in\liel^*$, then
$$
c_\lieq(V)\;=\;
\frac
{\sum_{w\in W(\lieg,\liel)}(-1)^{\ell(w)}[\CC_{w(\lambda+\rho(\lieu))-\rho(\lieu)}]}
{\sum_{w\in W(\lieg,\liel)}(-1)^{\ell(w)}[\CC_{w(\rho(\lieu))-\rho(\lieu)}]}
$$
and we have the denominator formula
$$
\sum_{w\in W(\lieg,\liel)}(-1)^{\ell(w)}[\CC_{w(\rho(\lieu))-\rho(\lieu)}]
\;\;\;=\!\!
\prod_{\alpha\in\Delta(\lieu,\liel)}({\bf1}-[\CC_{-\alpha}]).
$$
\end{theorem}

\begin{proof}
It is an immediate consequence of the definition of $c_\lieq$ and Kostant's Theorem, i.e.\ Theorem \ref{thm:kostant} and its Corollary \ref{cor:denominator}, together with the isomorphism \eqref{eq:tensorcomplex}, which tells us that
$$
H_\lieq({\bf1})\;=\;
\sum_{q=0}^{\dim\lieu}(-1)^q[\bigwedge^q\lieu^*]\;=\;
\prod_{\alpha\in\Delta(\lieu^*,\liel)}({\bf1}-[\CC_\alpha]).
$$
The last identity is a consequence of the isomorphism
$$
\bigwedge^q\lieu^*\;=\;
\bigwedge^q\bigoplus_{\alpha\in\Delta(\lieu^*,\liel)}\CC_\alpha\;\cong\;
\bigoplus_{\substack{A\subseteq\Delta(\lieu^*,\liel)\\\#A=q}}\bigotimes_{\alpha\in A}\CC_\alpha.
$$
\end{proof}

\subsection{Duality and Transitivity}

In this section we study the generalizations of \eqref{eq:fdres} and \eqref{eq:fddual}.

For transitivity and restrictions we consider the following setup. Fix an inclusion $i:(\lieh,N)\to(\lieg,K)$ of reductive pairs, and let $\liep\subseteq\lieh$ resp.\ $\lieq\subseteq\lieg$ be germane parabolic subpairs with $\liep\subseteq\lieq$ and compatible Levi decompositions $\liep=\liem+\lien$ resp.\ $\lieq=\liel+\lieu$, i.e.\ $\liem\subseteq\liel$. Note that then $\liep\cap\liel$ is germane parabolic in the reductive pair $(\lieh\cap\liel,N\cap L)$. Assume that for $?,!\in\{\rm df, \rm a, \rm fl, \rm fd\}$ restriction along $i$ defines functors $\mathcal F:\mathcal C_?(\lieg,K)\to\mathcal C_!(\lieh,N)$ and similarly $\mathcal F:\mathcal C_?(\liel,L\cap K)\to\mathcal C_!(\lieh\cap\liel,N\cap L)$. Then both forgetful functors descend to Grothendieck groups and respective localizations.

\begin{proposition}[Transitivity]\label{prop:transitivity}
In the setup above, we have the identity
$$
c_{\liep}\circ\mathcal F\;=\;
c_{\liep\cap\liel}\circ\mathcal F\circ c_{\lieq}\;\;\in\;\;
C_\liep(\liem,M\cap N)[W_\liep/W_\lieq|_{\liem,M\cap N}].
$$ 
\end{proposition}

There are two extreme cases. First $i$ may be the identity map, then we may choose $?=\,!$ and $\mathcal F$ become the identity functors. Then the statement reduces is equivalent to the (generalization of the) transitivity relation \eqref{eq:fdres} that we saw in the finite-dimensional case before. Secondly, if $i$ is not the identity, and $\lieh=\liel$ say, then Proposition \ref{prop:transitivity} boils down to Proposition \ref{prop:restriction}. In particular if $\lieh\cap\liel=\liem$, then Proposition \ref{prop:transitivity} tells us that characters commute with restrictions.

\begin{proof}
We'll put us in a position where we can apply a Hochschild-Serre spectral sequence. To this point we introduce the germane parabolic
$$
\liep'\;:=\;\lieq\cap\lieh\;\subseteq\;\lieh.
$$
It has a Levi decomposition $\liep'=\liel'+\lien'$ with $\liel'=\liel\cap\lieh$ and $\lien'\subseteq\lien$. In particular
$$
\lien\;=\;(\lien\cap\liel)\oplus\lien'\;=\;(\lien\cap\liel')\oplus\lien'.
$$
The Hochschild-Serre spectral sequence for the inclusion $\liel'\subseteq\lieh$ and the parabolic subalgebras $\liep\cap\liel\subseteq\liep\subseteq\liep'$ gives us the relation
\begin{equation}
H_\liep(X)\;=\;H_{\liep\cap\liel}\circ H_{\liep'}(X)\;\;\in\;\; K_!(\liem,M\cap N),
\label{eq:pprime}
\end{equation}
for all $X\in K_!(\lieh,N)$.

Now we can also apply the Hochschild-Serre spectral sequence to the inclusion $\lieh\subseteq\lieg$ and the parabolic algebras $\liep'\subseteq\lieq$. From this we get
$$
\sum_{q=0}^{\dim\lieu/\lien'}[\bigwedge^q(\lieu/\lien')^*]\cdot
H_{\liep'}\circ\mathcal F(X)\;=\;\mathcal F\circ H_{\lieq}(X)\;\;\in\;\; K_?(\liel',L'\cap K),
$$
for all $X\in K_?(\lieg,K)$. The sum on the left is the quotient of the Weyl denominators $W_\lieq|_{\liel',L'\cap K}$ and $W_\liep'$. We conclude that
\begin{equation}
H_{\liep'}\circ\mathcal F(X)\cdot W_\lieq|_{\liel',L'\cap K}\;=\;
\mathcal F\circ H_{\lieq}(X)\cdot W_{\liep'}
\;\;\in\;\; K_?(\liel',L'\cap K),
\label{eq:wpq}
\end{equation}
Plugging \eqref{eq:pprime} and \eqref{eq:wpq} together we obtain
$$
H_\liep\circ\mathcal F(X)\cdot W_\lieq|_{\liem,M\cap N}\;=\;
H_{\liep\cap\liel}\circ H_{\liep'}\circ\mathcal F(X)\cdot W_\lieq|_{\liem,M\cap N}\;=\;
$$
$$
H_{\liep\cap\liel}\circ \mathcal F\circ H_{\lieq}(X)\cdot W_{\liep'}|_{\liem,M\cap N}
\;\;\in\;\; K_!(\liem,M\cap N),
$$
for all $X\in K_?(\lieg,K)$, where we implicitly exploited the multiplicativity relation
$$
H_{\liep\cap\liel}((-)\cdot W_\lieq|_{\liel',L'\cap N})\;=\;
H_{\liep\cap\liel}(-)\cdot W_\lieq|_{\liem,M\cap N}.
$$
Finally
$$
W_\liep\;=\;W_{\liep\cap\liel}\cdot W_{\liep'}|_{\liem,M\cap N}
$$
concludes the proof.
\end{proof}

Our hypotheses for Proposition \ref{prop:transitivity} are stricter than necessary. The assumption that $\mathcal F$ be defined on the entire Grothendieck groups may be weakened. Not only other Grothendieck groups may be considered, but also suitable subgroups of those groups. This is made precise by the notion of {\em admissible quadruple} in \cite{januszewskipreprint}, which allows a more general construction of characters. However the methods and main arguments are still the same as the ones here.

There are two important consequences of Proposition \ref{prop:transitivity} that are also reflected in its proof. The first is the transitivity relation \eqref{eq:fdres}, which reduces many properties of characters, including their explicit calculation, to the case of maximal parabolic subalgebras. The second is the compatibility with restrictions, which is easiest expressed in the case $\liep\cap\liel=\liem$, which amounts to the commutativity of the square
$$
\begin{CD}
K_?(\lieg,K)@>\mathcal F>> K_!(\lieh,N)\\
@V{c_\lieq}VV @V{c_\liep}VV\\
C_?(\liel,L\cap K)@>\mathcal F>>C_!(\liem,M\cap N)[W_\liep/W_\lieq|_{\liem,M\cap N}]
\end{CD}
$$
We remark that by our assumptions we have in this situation $\lien\subseteq\lieu$, which means that we get a natural isomorphism
$$
C_!(\liem,M\cap N)[W_\liep/W_\lieq|_{\liem,M\cap N}]\;\to\;K_!(\liem,M\cap N)[W_\lieq|_{\liem,M\cap N}^{-1}].
$$
In particular for a fixed reductive pair $(\liem,M\cap N)$ we are led to consider many different localizations, depending on the branching problem at hand. Therefore we are naturally led to study the problem of localization in general.

The main issue being that $c_\liep$ (and even $c_\lieq$) in the above diagram can be far from injective. One reason is the possible vanishing of cohomology. The other is the non-triviality of the kernel of the localization map. We will see that there is a connection between these two seemingly independent situations in our treatment of Blattner formulae.

The third important consequence of Proposition \ref{prop:transitivity} resp.\ its proof is
\begin{corollary}\label{cor:constructiblefinitelength}
Let $\lieq=\liel+\lieu\subseteq\lieg$ be constructible. Then $\lieu$-cohomology preserves finite length.
\end{corollary}

As for the duality, we observe that the isomorphism \eqref{eq:tensorcomplex} applied to the trivial module $X={\bf1}=\CC$, together with Poincar\'e duality gives
$$
(-1)^{\dim\lieu}[\bigwedge^{\dim\lieu}\lieu^*]\cdot
\sum_{q=0}^{\dim\lieu}[\bigwedge^q\lieu]\;=\;
\sum_{q=0}^{\dim\lieu}[\bigwedge^q\lieu^*].
$$
The sum on the left hand side computes cohomology of ${\bf1}$, and the right hand side computes its cohomology. We deduce that
$$
W^\lieq\;:=\;\sum_{q=0}^{\dim\lieu}(-1)^qH_q(\lieu;{\bf1})\;\in\;K_{\rm fd}(\liel,L\cap K)
$$
acts invertibly (via multiplication) in a given $K_{\rm fd}(\liel,L\cap K)$-module $M$ if and only if $W_\lieq$ acts invertibly. By Proposition \ref{prop:easyduality} or direct inspection
$$
W_\lieq^\vee\;=\;W^\lieq.
$$
Therefore
\begin{equation}
(-1)^{\dim\lieu}[\bigwedge^{\dim\lieu}\lieu^*]\cdot
H_{\lieq}({\bf1})^\vee\;=\;
H_{\lieq}({\bf1}).
\label{eq:trivialcdual}
\end{equation}
This justifies that dualization extends to localizations, and we may formulate
\begin{proposition}[Duality]\label{prop:duality}
Let $X,X^\vee\in\mathcal C_?(\lieg,K)$ be two admissible modules
and assume $\lieq$ to be a $\theta$-stable Borel.
Then we have the identity
$$
c_\lieq(X^\vee)\;=\;c_\lieq(X)^\vee
$$
in $C_{\lieq,?}(\liel,L\cap K)$, where on the left hand side $\cdot^\vee$ denotes the $K$-finite dual, and on the right hand side the $(L\cap K)$-finite dual.
\end{proposition}

The restriction to finiteness seems necessary, as our proof relies on the reflexivity on $X$. For finite-length modules we obtain the same statement along the lines of the proof of Theorem \ref{thm:independence} below.

\begin{proof}
The crucial point is the comparison between the $K$-finite and the $L\cap K$-finite duals. Therefore we write $X^{\vee,G}$ for the $K$-finite and $X^{\vee,L}$ for the $L\cap K$-finite duals respectively. Our goal is to show that the inclusion
$$
X^{\vee,G}\;\to\;
X^{\vee,L}
$$
induces an isomorphism
\begin{equation}
H^q(\lieu; X^{\vee,G})\cong
H^q(\lieu; X^{\vee,L})
\label{eq:gldualmap}
\end{equation}
in cohomology. Then the claim of the proposition will follow from Poincar\'e duality. The biduality maps induce short exact sequences
\begin{equation}
\begin{CD}
0@>>> X@>\nu_G>> (X^{\vee,G})^{\vee,G}@>>> 0 @>>> 0
\end{CD}
\label{eq:xgbidual}
\end{equation}
\begin{equation}
\begin{CD}
0@>>> X@>\nu_{GL}>> (X^{\vee,G})^{\vee,L}@>>> Y @>>> 0
\end{CD}
\label{eq:xglbidual}
\end{equation}
\begin{equation}
\begin{CD}
0@>>> X@>\nu_L>> (X^{\vee,L})^{\vee,L}@>>> N @>>> 0
\end{CD}
\label{eq:xlbidual}
\end{equation}
of $(\lieq,L\cap K)$-modules. We also have a short exact sequence
\begin{equation}
\begin{CD}
0@>>> Z@>>> (X^{\vee,L})^{\vee,L}@>>> (X^{\vee,G})^{\vee,L} @>>> 0.
\end{CD}
\label{eq:xgldual}
\end{equation}
Plugging these sequences together, we obtain a commutative diagram
$$
\begin{CD}
X@>\nu_{GL}>> (X^{\vee,G})^{\vee,L}@>>> Y\\
@AAA @A\eta AA @AAA\\
X@>\nu_L>> (X^{\vee,L})^{\vee,L}@>>> N\\
@AAA @AAA @AAA\\
 0@>>> Z@>>> Z\\
\end{CD}
$$
with short exact sequences in the rows and columns. The exactness of the last row resp.\ last column follows from the snake lemma.

By construction of the Poincar\'e duality map in cohomology (i.e.\ Corollary \ref{cor:duality}) the map $\nu_L$ induces the biduality map
$$
H^q(\lieu;X)\;\to\;H^q(\lieu;X)^{\vee,\vee}
$$
on cohomology, which is an isomorphism as $H^q(\lieu;-)$ preserves admissibility by Theorem \ref{thm:inheritance}. Then the long exact sequence of cohomology tells us that the $\lieu$-cohomology of $N$ vanishes in all degrees, and therefore the long exact sequence for the right most column gives us isomorphisms
$$
H^q(\lieu; Y)\;\to\;H^{q+1}(\lieu; Z).
$$
In order to see that \eqref{eq:gldualmap} is an isomorphism, it suffices to show that the cohomology of $Y$ and $Z$ vanishes.

To this point we consider the same diagrams over the associated compact pairs. Duality is not affected by this restriction, and we obtain analogously and isomorphism
$$
H^q(\lieu\cap\liek; Y)\;\to\;H^{q+1}(\lieu\cap\liek; Z).
$$
Together with Kostant's Theorem this implies the vanishing of the $\lieu\cap\liek$-cohomology of $Y$ and $Z$ in the light of Proposition \ref{prop:kdecomp}. We have a Hochschild-Serre spectral sequence
$$
\bigwedge^p(\lieu\cap\liep)^*\otimes_\CC H^q(\lieu\cap\liek; Y)
\;\Longrightarrow\;H^{p+q}(\lieu; Y),
$$
where $\liep$ denotes the orthogonal complement to $\liek$ in $\lieg$. As the left hand side vanishes, the right hand side does so too, and we conclude the proof.
\end{proof}

The proof generalizes to arbitrary $\theta$-stable parabolics by appealing to the general version of Kostant's Theorem, cf.\ \cite{januszewskipreprint}.

\subsection{Linear independence}

In this section we assume that the associated Lie group $G$ to $(\lieg,K)$ is linear, which is a mild condition needed for Hecht-Schmid's result below. In order to proof that algebraic characters classify the semi-simplifications of finite-length modules we import the following two results.

The following result was conjectured by Osborne in his thesis \cite{osborne1972} for Borel algebras and generalized to the general real parabolic case by Casselman \cite{casselman1977}.

\begin{proposition}[Hecht-Schmid, \cite{hechtschmid1983}]\label{prop:osbornereal}
Let $\lieq\subseteq\lieg$ be a real parabolic subalgebra with Levi decomposition $\lieq=\liel+\lieu$. Then there is an open subset $L_0\subseteq L\subseteq G$ which contains the set of regular elements of $L$, and which maps surjectively onto $L^{\rm ad}$ under the adjoint map, such that for any finite-length $(\lieg,K)$-module $X$ Harish-Chandra's global characters for $G$ (cf.\ \cite{harishchandra1954b}) resp.\ $L$ satisfy the identity
$$
\Theta_{G}(X)|_{L_0}\;=\;
\frac{\sum_{q=0}^{\dim\lieu}(-1)^q\Theta_{L}(H^q(\lieu;X))|_{L_0}}
{\sum_{q=0}^{\dim\lieu}(-1)^q\Theta_{L}(H^q(\lieu;{\bf1}))|_{L_0}}.
$$
\end{proposition}

The set $L_0$ from Proposition \ref{prop:osbornereal} is large enough that the left hand side, and hence the right hand side, determines the restriction of $\Theta_G(X)$ to $L$ uniquely.

\begin{proposition}[Vogan, {\cite[Theorem 8.1]{vogan1979ii}}]\label{prop:osbornethetastable}
Let $\lieq\subseteq\lieg$ be a $\theta$-stable parabolic subalgebra with Levi decomposition $\lieq=\liel+\lieu$. Then for any finite-length $(\lieg,K)$-module $X$ Harish-Chandra's global characters for $G$ resp.\ $L$ satisfy the identity
$$
\Theta_{G}(X)|_{L}\;=\;
\frac{\sum_{q=0}^{\dim\lieu}(-1)^q\Theta_{L}(H^q(\lieu;X))}
{\sum_{q=0}^{\dim\lieu}(-1)^q\Theta_{L}(H^q(\lieu;{\bf1}))}.
$$
\end{proposition}

Now we can proof

\begin{theorem}[Linear independence]\label{thm:independence}
Assume that $\lieq_1,\dots,\lieq_r$ is a collection of constructible parabolic subalgebras
whose Levi factors $L_1,\dots,L_r$ cover a dense subset of $G$ up to conjugation. Then the map
$$
\prod_{i=1}^rc_{\lieq_i}:\;\;\;K_{\rm fl}(\lieg,K)\;\to\;\prod_{i=1}^r C_{\lieq_i}(\liel_i,L_i\cap K)
$$
is a monomorphism.
\end{theorem}

\begin{proof}
By Proposition \ref{prop:parabolic}, the transitivity relation from Proposition \ref{prop:transitivity}, and the above Propositions \ref{prop:osbornereal} and \ref{prop:osbornethetastable} we know that each character
$$
c_{\lieq_i}(X)\;\in\;C_{\lieq_i}(\liel_i,L_i\cap K)
$$
determines the restriction of Harish-Chandra's global Character $\Theta_G(X)$ of $X$ to $L_i$ uniquely. Therefore by our assumption that the conjugates of $L_i$ cover a dense subset of $G$ we deduce the claim of the Theorem from Harish-Chandra's linear independence Theorem for his global characters \cite[Theorem 1]{harishchandra1954b}.
\end{proof}

\section{Localization}

In order to extend Theorem \ref{thm:independence} to more general than finite-length modules we need to study localization in more detail.

\subsection{An example}

Let us consider $G=\SL_2(\RR)$ and the corresponding reductive pair $(\liesl_2,\SO(2))$. We choose $\lieq=\liel+\lieu$ a minimal $\theta$-stable parabolic with a Levi decomposition as indicated.

Then $K_{\rm fl}(\liesl_2,\SO(2))$ is the free abelian group generated by the irreducible $(\liesl_2,\SO(2))$-modules, which are all known. For the Levi factor we have
$$
L\cap\SO(2)\;=\;L\;=\;\SO(2),
$$
and the corresponding Grothendieck group
$$
K_{\rm fl}(\liel,L)\;=\;K_{\rm fd}(\liel,L)
$$
is the free abelian group generated by the irreducible (hence finite-dimensional) representations of $\SO(2)$. In this group we have the Weyl denominator
$$
W_\lieq:=1-[2\alpha],
$$
where $\alpha$ is a generator of all the characters of $\SO(2)$, compatible with our choice of $\lieu$, i.e.\ the weight occuring in $\lieu$ is $-2\alpha$. As we know that $K_{\rm fd}(\liel,L)$ is a domain we know that localization at $W_\lieq$ is a faithful operation.

We may identify $K_{\rm fd}(\liel,L)$ as a ring with the ring
$\ZZ[X,X^{-1}]$, where $X$ per definition corresponds to $\alpha$, and we may think of
$$
C_{\lieq,\rm fd}(\liel,L)\;=\;K_{\rm fd}(\liel,L)[W_\lieq^{-1}]
$$
as the subring
$$
\ZZ[X,X^{-1},\frac{1}{1-X^2}]\;\subseteq\;\QQ(X).
$$
of the rational function field $\QQ(X)$. The character of a finite-dimensional $(\liesl_2,\SO(2))$-module $F_k$ of $\lieu$-highest weight $k\geq 0$ is, by Weyl's formula,
$$
c_\lieq(F_k)\;=\;\frac{X^k-X^{-(k+2)}}{1-X^2}.
$$
The formal identity
$$
\frac{X^k-X^{-(k+2)}}{1-X^2}\;=\;X^k+X^{k-2}+\cdots+X^{-(k-2)}+X^{-k}
$$
reflects the second identity of Theorem \ref{thm:fdmono}, i.e.\
\begin{equation}
c_\lieq(F_k)\;=\;[F_k|_{\liel,L}]\;=\;[k\alpha]+[(k-2)\alpha]+\cdots+[-(k-2)\alpha]+[-k\alpha].
\label{eq:finitelres}
\end{equation}
In particular we get the structure of $F_k$ as a $(\liel,L)$-module. A fundamental property of cohomological characters is that this remains true in general, not only for finite-dimensional modules, but whenever both sides of the equation happen to be well defined.

Now write $D_k$ for the (limits of) discrete series repesentation with lowest $\SO(2)$-type $k\cdot\alpha$, $k\geq 1$. Its character is
$$
c_\lieq(D_k)\;=\;\frac{X^k}{1-X^2}.
$$
Assuming $k\geq k'+3$, the multiplicativity of our characters tells us that
$$
c_{\lieq}(D_k\otimes_\CC F_{k'})\;=\;c_\lieq(D_k)\cdot c_\lieq(F_{k'})\;=\;
\frac{X^k}{1-X^2}\cdot\frac{X^{k'}-X^{-({k'}+2)}}{1-X^2}\;=\;
$$
$$
\frac{X^k\cdot(X^{k'}+X^{k'-2}+\cdots+X^{-(k'-2)}+X^{-k'})}{1-X^2}\;=\;
$$
$$
c_\lieq(D_{k+k'})\;+\;c_\lieq(D_{k+k'-2})\;+\;\cdots\;+\;c_\lieq(D_{k-k'-2}).
$$
Without appealing to Theorem \ref{thm:independence} we would like to argue that this formula implies the decomposition relation
\begin{equation}
[D_k\otimes_\CC F_{k'}]\;=\;[D_{k+k'}]\;+\;[D_{k+k'-2}]\;+\;\cdots\;+\;[D_{k-k'-2}]
\label{eq:tensordecomp}
\end{equation}
in $K_{\rm fl}(\liesl_2,\SO(2))$. To do so we first solve the Blattner problem for each $D_k$.

With the same notation as before, we have canonical isomorphisms
$$
K_{\rm df}(\liel,L)\;=\;K_{\rm a}(\liel,L)\;=\;\ZZ[[X,X^{-1}]]
$$
of abelian groups, where the right hand side denotes the {\em abelian group} of formal unbounded Laurent series in $X$. It is naturally a $\ZZ[X,X^{-1}]$-module, and this module structure is compatible with the canonical $K_{\rm fd}(\liel,L)$-module structure of the left hand side. The restriction of $D_k$ to $(\liel,L)$ will be an element in this Grothendieck group. However, due to the compatibility of characters with restrictions, we get in the localization the relation
$$
\frac{X^k}{1-X^2}\;=\;[D_k|_{\liel,L}]\;\in\;\ZZ[[X,X^{-1}]][\frac{1}{1-X^2}].
$$
Now we have
$$
\frac{X^k}{1-X^2}\;=\;\sum_{i=0}^\infty X^{k+2i},
$$
and we would like to conclude that
\begin{equation}
[D_k|_{\liel,L}]\;=\;[k\alpha]+[(k+2)\alpha]+[(k+4)\alpha]+\cdots
\label{eq:unlocalkl}
\end{equation}
is a valid identity in the unlocalized $K_{\rm df}(\liel,L)$. The problem is that this identity is valid a priori only up to the kernel of the localization map
\begin{equation}
\ZZ[[X,X^{-1}]]\;\to\;\ZZ[[X,X^{-1}]][\frac{1}{1-X^2}].
\label{eq:localizationx}
\end{equation}
It is an easy excercise to see that the kernel is, as a vector space, generated by the elements
\begin{equation}
y_\alpha^{(n)}\;:=\;\sum_{i=0}^\infty\binom{n-1+i}{n-1}(X^{n+2i}\;+\;(-1)^{n+1}X^{-(n+2i)})
\label{eq:yalphan}
\end{equation}
and
\begin{equation}
y_{\alpha,+}^{(n)}\;:=\;\sum_{i=0}^\infty\frac{n-1+2i}{n-1+i}\binom{n-1+i}{n-1}(X^{n-1+2i}\;+(-1)^{n+1}\;X^{-(n-1+2i)})
\label{eq:yalphanplus}
\end{equation}
for $n\geq 1$. We remark that the coefficients occuring eventually are {\em integers}. Note that
$$
(X+X^{-1})\cdot y_\alpha^{(n)}\;=\;y_{\alpha,+}^{(n)},
$$
and
$$
(X-X^{-1})\cdot y_\alpha^{(n)}\;=\;\;y_\alpha^{(n-1)}.
$$
Therefore, by induction, the kernel of $(X-X^{-1})^n$ as an endomorphism of $\ZZ[[X,X^{-1}]]$ is generated by $y_\alpha^{(n)}$ as a $\ZZ[X,X^{-1}]$-module. Using the above relations it is not hard to see, and appealing to Proposition \ref{prop:lockernel}, that the collection of the elements \eqref{eq:yalphan} and \eqref{eq:yalphanplus} for $n\geq 1$ are indeed a $\CC$-basis of kernel of the map \eqref{eq:localizationx}.

In order to proof \eqref{eq:unlocalkl}, Harish-Chandra tells us that the multiplicity of the $\SO(2)$-module $m\alpha$ in $D_k$ is bounded by a constant independently of $m$. However the coefficients in \eqref{eq:yalphan} and \eqref{eq:yalphanplus} grow polynomially of order $n-1$, and this remains true for linear combinations of those terms. Therefore the only possible contribution from the kernel of the localization map to the identity \eqref{eq:unlocalkl} is a $\CC$-linear combination of
$$
y_\alpha^{(1)}\;=\;\sum_{i=-\infty}^\infty X^{1+2i}
$$
and
$$
y_{\alpha,+}^{(1)}\;=\;\sum_{i=-\infty}^\infty X^{2i}.
$$
However the minimal $\SO(2)$-type of $D_k$ is $k\cdot\alpha$ and is uniquely determined. As to $k\geq 1$, this implies that there cannot be any contribution from those two kernel elements to the identity \eqref{eq:unlocalkl}, i.e.\ the latter formula must be true, completing its proof.

With \eqref{eq:unlocalkl} at hand, we may solve the identity \eqref{eq:tensordecomp} explicitly in the unlocalized Grothendieck group of admissible $(\liel,L)$-modules:
$$
[D_k|_{\liel,L}\otimes_\CC F_{k'}|_{\liel,L}]\;=\;\sum_{i=0}^\infty X^{k+2i}\cdot(X^{k'}+X^{k'-2}+\cdots+X^{-k'})\;=\;
$$
$$
[D_{k+k'}|_{\liel,L}]\;+\;[D_{k+k'-2}|_{\liel,L}]\;+\;\cdots\;+\;[D_{k-k'-2}|_{\liel,L}].
$$
by the decomposition relation \eqref{eq:finitelres} for $F_{k'}$. We see that over $L$, there is no contribution from the kernel of the localization to the identity \eqref{eq:tensordecomp}. A fortiori this is true over $(\liesl_2,\SO(2))$, hence \eqref{eq:tensordecomp} follows.

Along the same lines we may deduce from
$$
c_\lieq(D_k\otimes_\CC D_{k'})\;=\;c_\lieq(D_k)\cdot c_\lieq(D_{k'})
$$
the decomposition of the tensor product of two discrete series modules $D_k$ and $D_{k'}$. It decomposes again into a sum of discrete series modules.

So what happens if we consider the opposite discrete series $D_{-k}$, i.e.\ the unique irreducible $(\liesl_2,SO(2))$-module with (unique) minimal $\SO(2)$-type $-k\alpha$? Here the tensor product of $D_k$ and $D_{-k'}$ gives along the same lines
$$
c_\lieq(D_k\otimes_\CC D_{-k'})\;=\;c_\lieq(D_k)\cdot c_\lieq(D_{k'})\;=\;
\frac{X^k}{1-X^2}\cdot\frac{X^{-k-2}}{1-X^2}.
$$
This is meaningful in the localization, but cannot capture everything here. The reason being that the tensor product of $D_k$ and $D_{-k'}$ as modules over $\SO(2)$, lies no more in $\mathcal C_{\rm df}(\liel,L)$, because the multiplicity of either $0\cdot\alpha$ or $1\cdot\alpha$, depending on the parity of $k+k'$, is no more finite. This tells us that there must be a (large) contribution of the continuous spectrum to $D_k\otimes_\CC D_{-k'}$.

Another interesting relation is
$$
[D_1]-[D_{-1}]\;=\;y_{\alpha}^{(1)},
$$
which shows that this element lies in the kernel of the localization map \eqref{eq:localizationx}. It is annihilated by $W_\lieq=1-X^2$. However there is no such relation for finite linear combinations of elements $[D_k]$ and $[D_{-k'}]$ if we insist on $k,k'\geq 2$. Because we know that asymptotically the multiplicity of the $\SO(2)$-types in such a virtual finite length representation are bounded by a constant. Therefore the only contribution of the kernel of the localization map comes from the kernel of $W_\lieq=1-X^2$ itself, which is generated by $y_\alpha^{(1)}$ and $y_{\alpha,+}^{(1)}$ as a $\CC$-vector space. This shows that any such element from the kernel must contain at least one of the two $\SO(2)$-modules $0\cdot\alpha$ or $1\cdot\alpha$. In particular this abtract argument shows that the subgroup
$$
K_{\rm fl}^{\geq 2}(\liesl_2,\SO(2))\;\subseteq\;
K_{\rm fl}(\liesl_2,\SO(2))
$$
generated by $D_k$ and $D_{-k'}$ for $k,k'\geq 2$, has trivial intersection with the kernel \eqref{eq:localizationx}, a fortiori its intersection with the kernel of $c_\lieq$ is trivial.

An analogous statement remains true if we consider the subgroup
$$
K_{\rm df}^{\geq d+2,d}(\liesl_2,\SO(2))\;\subseteq\;
K_{\rm df}(\liesl_2,\SO(2))
$$
of discretely decomposables with finite multiplicity, which is generated by those with the property that the multiplicity of $D_k$ resp.\ $D_{-k'}$ is bounded by a polynomial of degree $d$ in $k$ resp.\ $k'$, and which do {\em not} contain $D_k$ resp.\ $D_{-k'}$ for $k,k'<d+2$. This again is a consequence of the explicit form of the elements \eqref{eq:yalphan} and \eqref{eq:yalphanplus}.

\subsection{The general case}

As we already saw in the examples in the previous section, we need to understand the kernel of the localization map in order to lay hands on larger-than-finite-length modules. This situation already arises when we approach Blattner formulae, i.e.\ restrictions of admissible modules from $(\lieg,K)$ to $K$. Generally the restriction will no more be of finite length, as over $K$ finite length is equivalent to finite dimension.

We will give complete results here, but won't go into detailed proofs, as this is quite technical. For details the reader can consult Sections 4 to 6 in \cite{januszewskipreprint}.

In principle one may consider the localization problem for general germane parabolic subalgebras, i.e.\ for Levi factors whose Lie algebras are not necessarily abelian. However in this case the structure of the Grothendieck group in question, $K_{\rm df}(\liel,L\cap K)$ say, as a $K_{\rm fd}(\liel,L\cap K)$-module is not clear in general.

In any case the module structure of $K_{\rm df}(\liel,L\cap K)$ may be deduced formally from the $K_{\rm fd}(\liel,L\cap K)$-module structure of $K_{\rm fl}(\liel,L\cap K)$.

Classically characters are studied via their restrictions to Levi factors corresponding to Cartan algebras $\liel$. This is also the case where we can give complete answers as then the multiplicative structure alluded to above is clear. Another appeal of this case is that the $\lieu$-cohomology of any finite length $(\lieg,K)$-module is a {\em finite-dimensional} $(\liel,L\cap K)$-module.

Let us introduce the basic setup. We assume given a commutative square
$$
\begin{CD}
(\lieg',K')@>i>>(\lieg,K)\\
@AAA @AAA\\
(\lieq',L'\cap K') @>j>> (\lieq,L\cap K)
\end{CD}
$$
of pairs, where $i$ is an inclusion of reductive pairs, and the vertical arrows are inclusions of constructible parabolic pairs, and $j$ is induced by $i$. Let $\lieq=\liel+\lieu$ and $\lieq'=\liel'+\lieu'$ be the Levi decompositions and assume that $\liel'\subseteq\liel$ is abelian, i.e.\ a Cartan subalgebra in $\lieg'$.

Then $\Delta(\lieu+\lieu',\liel')$ generates a lattice
$$
\Lambda_0\;\subseteq\;\liel'^*.
$$
The group algebra $\CC[\Lambda_0]$ acts naturally on the space $\CC[[\Lambda_0]]$ of formal unbounded Laurent series in $\Lambda_0$ via multiplication. This turns $\CC[[\Lambda_0]]$ into a $\CC[\Lambda_0]$-module which reflects the $K_{\rm fd}(\liel',L'\cap K')$-module structure of $K_{\rm df}(\liel',L'\cap K')$ with the restriction that $\Lambda_0$ only captures $\lieu+\lieu'$-integral weights. In our study of localization this is enough, as the Weyl denominator stemming from $\lieu$ and $\lieu'$ naturally lives inside $\CC[\Lambda_0]$, and its kernel in $\CC[[\Lambda_0]]$ is sufficient to describe the kernel in the larger $K_{\rm df}(\liel',L'\cap K')$.

However for reasons of symmetry we introduce a slightly larger lattice $\Lambda$. Say the rank of $\Lambda_0$ is $d$, then $\Lambda$ is uniquely characterized by the short exact sequence
$$
0\to\Lambda_0\to\Lambda\to(\ZZ/2\ZZ)^d\to 0.
$$
We think of $\Lambda$ as a {\em multiplicative} abelian group. Then any $\alpha\in\Delta(\lieu+\lieu',\liel')$ has a {\em square root} $\alpha^{\frac{1}{2}}$ in $\Lambda$, i.e.\ and element whose square is $\alpha$. Further to any such $\alpha$ we have a generalized {\em Weyl reflection} $w_\alpha$ sending $\alpha$ to $\alpha^{-1}$ and $\alpha^{\frac{1}{2}}$ to $\alpha^{-\frac{1}{2}}$.

To define $w_\alpha$ we choose a Cartan subalgebra $\lieh\subseteq\liel$ containing $\liel'$ complementary to $\lieu+\lieu'$. Then $\alpha$ has a preimage $\alpha'\in\Delta(\lieu+\lieu',\lieh)$ and the associated Weyl reflection $w_{\alpha'}$ projects down to $w_\alpha$ in $\Lambda$. This projected image is independent of the choice of $\alpha'$. The collection of $w_\alpha$ generate a generalized Weyl group $W_\Lambda$ which contains the Weyl group $W(\lieg',\liel')$. Both groups act on $\CC[\Lambda]$ and $\CC[[\Lambda]]$. We fix a $W_\Lambda$-invariant bilinear scalar product $\langle\cdot,\cdot\rangle$ on $\Lambda\otimes_\ZZ\RR$.

In order to descend from $\Lambda$ to $\Lambda_0$ we introduce the Galois group $G_{\Lambda/\Lambda_0}$ of the field extension $\CC(\Lambda)/\CC(\Lambda_0)$, the latter fields being the quotient fields of the domains $\CC[\Lambda]$ resp.\ $\CC[\Lambda_0]$. By construction $G_{\Lambda/\Lambda_0}$ is isomorphic to $\{\pm 1\}^d$, each sign corresponding to a choice of square root of an element of a fixed basis of $\Lambda_0$ in $\CC[\Lambda]$. Then $G_{\Lambda/\Lambda_0}$ acts on $\CC[\Lambda_0]$ and $\CC[\Lambda]$ and this action commutes with the action of $W_\Lambda$. Then we may use standard Galois theory to descend from $\Lambda$ to $\Lambda_0$. We leave the details to the reader.

In the higher rank case the kernel of the localization map is completely described by the following generalization of the elements \eqref{eq:yalphan} and \eqref{eq:yalphanplus}. For any $\alpha\in\Delta(\lieu+\lieu',\liel')$ with square root $\alpha^{\frac{1}{2}}\in\Lambda$ we have the elements
$$
d_{\alpha,\pm}\;\;:=\;\;\alpha^{-\frac{1}{2}}\pm\alpha^{\frac{1}{2}}
\;\;\in\;\;\CC[\Lambda],
$$
$$
s_{\alpha}\;\;:=\;\;
\alpha^{\frac{1}{2}}
\sum_{k=0}^\infty 
\alpha^{k}
\;\;\in\;\;\CC[[\Lambda]],
$$
and for $n\geq 0$ we set
$$
y_{\alpha}^{(n)}\;\;:=\;\;s_{\alpha}^n+(-1)^{n+1} w_\alpha s_{\alpha}^n.
$$
This is the one-dimensional generalization of the element \eqref{eq:yalphan}. This works out as
$$
d_{\alpha,-}\cdot s_{\alpha}\;=\;1,
$$
$$
d_{\alpha,-}\cdot w_\alpha s_{\alpha}\;=\; -1,
$$
gives
$$
d_{\alpha,-}\cdot y_{\alpha}^{(n)}\;=\;y_{\alpha}^{(n-1)},
$$
and therefore
$$
d_{\alpha,-}^n\cdot y_{\alpha}^{(n)}\;=\;y_{\alpha}^{(0)}\;=\;0.
$$
The second element \eqref{eq:yalphanplus} corresponds to the product of $y_{\alpha}^{(n)}$ and $d_{\alpha,+}$.

Say we are given elements $\beta_1,\dots,\beta_r\in\Lambda$. Then we define the module of {\em $(\beta_1,\dots,\beta_r)$-finite} Laurent series $\CC[[\Lambda]]_{(\beta_1,\dots,\beta_r)}$ as the space of series
$$
f\;=\;\sum_{\mu\in\Lambda}f_\mu\cdot\mu\;\;\in\;\;\CC[[\Lambda]],
$$
where $f_\mu\in\CC$, satisfying the finiteness following condition: For any $\lambda\in\Lambda$ the set
$$
\{(k_1,\dots,k_r)\in\ZZ^r\mid f_{\lambda\mu_1^{k_1}\cdots\beta_r^{k_r}}\neq 0\}
$$
is {\em finite}. Then $\CC[[\Lambda]]_{(\beta_1,\dots,\beta_r)}$ is a $\CC[\Lambda]$-submodule of $\CC[[\Lambda]]$ and it turns out that the kernel of multiplication with $d_{\alpha,-}^n$ in $\CC[[\Lambda]]$ is given by
$$
\CC[[\Lambda]]_{(\alpha_0)}\cdot y_{\alpha}^{(n)}.
$$
The above notion of finiteness is necessary for the products occuring in this expression to be well defined. This representation is still redundant, as the case $n=1$ already shows. Pick a system of representatives $\CC[[\Lambda]]_{\alpha^{\frac{1}{2}}=1}\subseteq\CC[[\Lambda]]_{(\alpha)}$ for the factor module
$$
\CC[[\Lambda]]_{(\alpha)}/\left(\CC[[\Lambda]]_{(\alpha)}\cdot (\alpha^{\frac{1}{2}}-1)\right).
$$
Then the above kernel equals
$$
\sum_{k=1}^n
\CC[[\Lambda]]_{\alpha^{\frac{1}{2}}=1}\cdot y_{\alpha}^{(k)}\;+\;\CC[[\Lambda]]_{\alpha^{\frac{1}{2}}=1}\cdot d_{\alpha,+}\cdot y_{\alpha}^{(k)},
$$
and this representation is unique.

The kernel of mixed terms
$$
d_{\underline{\alpha}}\;:=\;\prod_{i=1}^r d_{\alpha_i,-}^{n_i}
$$
for a collection
$$
\underline{\alpha}\;=(\alpha_1,\dots,\alpha_r)\;\in\;\Delta(\lieu,\liet)
$$
of pairwise distinct elements and a tuple of positive integer exponents
$$
\underline{n}=(n_1,\dots,n_r)
$$
is generated by the element
\begin{equation}
y_{\underline{\alpha}}^{\underline{n}}\;:=\;
s_{\alpha_1}^{n_1}\cdot s_{\alpha_2}^{n_2}\cdots s_{\alpha_r}^{n_r}
+
(-1)^{1+n_1+\cdots+n_r}
(w_{\alpha_1}s_{\alpha_1})^{n_1}\cdot (w_{\alpha_2}s_{\alpha_2})^{n_2}\cdots (w_{\alpha_r}s_{\alpha_r})^{n_r},
\label{eq:generaly}
\end{equation}
and its lower degree analogues. For a subset $I\subseteq\{1,\dots,r\}$ we denote by $\underline{\alpha}^I$ the tuple deduced from $\underline{\alpha}$ by deletion of the components with indices in $I$. Then again we have the explicit description of the kernel as
$$
\sum_{I\subsetneq\{1,\dots,r\}}
\CC[[\Lambda]]_{\underline{\alpha}^I}\cdot
y_{\underline{\alpha}^I}^{\underline{n}^I}
$$
and one may again find an explicit description via a suitable set of representatives for the relations
$$
\forall i\not\in I:\;\;\;\alpha_i^{\frac{1}{2}}\;=\;1.
$$
However the resulting expression is technical due to its recursive structure resulting from lower degree terms.

Along these lines one may proof
\begin{theorem}[Januszewski, {\cite{januszewskipreprint}}]\label{thm:kernel}
Let
$$
z=\sum_{\lambda}z_\lambda\cdot\lambda
\;\in\;
\kernel d_{\underline{\alpha}}^{\underline{n}}
$$
with the property that there exists a $\lambda_0\in\Lambda$ such that for any $1\leq i\leq r$ and any $\lambda\in\Lambda$ with
\begin{equation}
\left|
\langle \lambda-\lambda_0,\alpha_i\rangle
\right|
<
\frac{n_i+1}{2}\cdot\langle \alpha_i,\alpha_i\rangle
+
\sum_{j\neq i}\frac{n_j}{2}\cdot\langle \alpha_i,\alpha_j\rangle,
\label{regularity}
\end{equation}
we have
$$
z_\lambda = 0
$$
then
$$
z=0.
$$
\end{theorem}

\subsection{Applications to Blattner formulae}

Suppose we are in the situation where $(\lieg',K')=(\liek,K)$, and assume for simplicity that $K$ is connected. The main idea to calculate the kernel of the localization map in this case is that analogously to the explicit expressions \eqref{eq:yalphan} and \eqref{eq:yalphanplus} we get explicit expressions for the term \eqref{eq:generaly}, which then shows that the coefficients of the monomials grow polynomially with degree depending on $\underline{n}$.

At the same time we know that by Harish-Chandra's bound for the multiplicities of $K$-types in irreducibles, that $K$-types are bounded by their dimensions.

Therefore we have for any finite length $(\lieg,K)$-module $X$ that the multiplicity $m_\lambda$ of the $K$-type in $X$ with highest weight $\lambda$ is bounded by
\begin{equation}
m_\lambda\;\;\leq\;\;
C\cdot\!\!\!
\prod_{\beta\in\Delta(\lien,\liet)}\!\!
\langle\lambda+\rho(\lien),\beta\rangle,
\label{eq:weylbound}
\end{equation}
for a constant $C$ depending on $X$ but not on $\lambda$, which is a consequence of Weyl's dimension formula.

Then this bound yields a constraint on the exponents $\underline{n}$ which ultimately results in the following condition:
\begin{itemize}
\item[(S)] 
For any highest weight $\lambda$ of a $K$-type occuring in $X$ we have for any $w\in W(K,L')=W(\liek,\liel')$, and any numbering $\beta_1,\dots,\beta_{r}$ of the pairwise distinct elements of $\Delta(\lieu+\lieu';,\liel')$ and any non-negative integers $n_1,\dots,n_r$ with the property that for any $S\subseteq\{1,\dots,r\}$
\begin{equation}
\sum_{s\in S}n_s\;\leq\;\#\{\beta\in\Delta(\lieu',\liel')\mid\exists i\in S:\langle\beta_i,\beta\rangle\neq 0\},
\label{eq:nsum}
\end{equation}
the condition
\begin{equation}
|\langle w(\lambda+\rho(\lieu'))-\lambda_0,\beta_1\rangle|
\;\geq\;
\frac{1}{2}
\langle \beta_1,\beta_1\rangle
\;+\,
\sum_{i=1}^{r}
\frac{n_{i}+1}{2}
\langle \beta_1,\beta_i\rangle.
\label{eq:lambdaregularity}
\end{equation}
\end{itemize}
Where we fixed $\lambda_0$. A natural choice would be $\lambda_0=\rho(\lieu')$ for example.

Then if we consider the subgroup $K_{\lambda_0}$ of $K_{\rm fl}(\lieg,K)$ generated by the modules $X$ satisfying condition (S) we know by Theorem \ref{thm:kernel} and our multiplicity bound that in the commutative diagram
$$
\begin{CD}
K_{\lambda_0}@>\iota>> K_{\rm df}(\liek,K)\\
@Vc_\lieq VV@Vc_{\lieq'} VV\\
C_{\lieq,\rm fl}(\liel,L\cap K)@>\iota>> C_{\lieq',\rm df}(\liel',L')[W_{\lieq/\lieq'}^{-1}]
\end{CD}
$$
the map $c_{\lieq'}$ is injective on the image of $\iota$, the latter denoting the forgetful map along $i$. In particular the restriction of the algebraic character $c_\lieq(X)$ to the maximal torus $L'\subseteq T$ uniquely determines the Blattner formula for every $X$ satisfying condition (S), cf.\ \cite{januszewskipreprint}.

By contraposition we obtain that any $X$ in the kernel of the localization map must contain a $K$-type violating the regularity condition \eqref{eq:lambdaregularity}.

In summary we reduced Blattner formulae to the following assertions:
\begin{itemize}
\item[(B)] {\em Boundedness:} Multiplicities of the $K$-types occuring in $X$ are linearly bounded by their dimension (Harish-Chandra).
\item[(S)] {\em Sample:} Knowledge of the multiplicities for the set of $K$-types violating the regularity condition \eqref{eq:lambdaregularity}.
\end{itemize}

We will see in the next section that the same statement remains true for not necessarily compact $(\lieg',K')$.

We emphasize that (S) needs to be known {\em a priori}. An instance where this may be verified without using any character theory is given by Schmid's upper bound for the $K$-types for the discrete series \cite[Theorem 1.3]{schmid1975}. His result may be used to deduce the Blattner conjecture in many cases, and we conjecture that our regularity condition for $\lambda_0=\rho(\lieu')$ eventually covers all of the discrete series. A complete proof of the Blattner formula using classical character theory was given by Hecht and Schmid in \cite{hechtschmid1975}.

\subsection{Applications to discretely decomposable branching problems}

Let us return to the case of a general incluson $i:(\lieg',K')\to (\lieg,K)$ of reductive pairs, and corresponding germane parabolic subalgebras $\lieq'$ and $\lieq$ as before.

In general the restriction of a finite length $(\lieg,K)$-module $X$ will not be discretely decomposable as a $(\lieg',K')$-module. Furthermore even if $X$ is discretely decomposable, we do not have a universal bound for the multiplicities of composition factors at hand, contrary to the case of $\lieg'=\liek$ treated in the previous section.

Therefore we need to impose these conditions a priori. Let us consider in $K_{\rm fl}(\lieg,K)$ the subgroup $K'$ generated by the modules $X$ with the property that the restriction to $(\lieg',K')$ is discretely decomposable with finite multiplicities. Then pullback along $i$ provides us with a natural group homomorphism
$$
\iota:\;\;\;K'\to K_{\rm df}(\lieg',K').
$$
We would like to have a similar map on the Levi factors of our parabolic subalgebras. But in order to guarantee its existence and compatibility with $\iota$, we assume that we are trating {\em Borel} subalgebras, i.e.\ we assume we are given a collection $\lieq_1,\dots,\lieq_s\subseteq\lieg$ of constructible Borel subalgebras, whose Levi factors $L_i$ cover a dense subset of $G'\subseteq G$ up to conjugation, where $G'$ is the reductive subgroup of $G$ corresponding to the reductive subpair $(\lieg',K')$ of $(\lieg,K)$. Such a collection always exists by Corollary \ref{cor:constructible}. We set
$$
\lieq_i'\;:=\;\lieq_i\cap\lieg',
$$
and fix the compatible Levi decompositions
$$
\lieq_i\;:=\;\liel_i+\lieu_i,
$$
and
$$
\lieq_i'\;:=\;\liel_i'+\lieu_i',
$$
respectively. Then all the Lie algebras $\liel_i$ are abelian and therefore restriction along the inclusion
$$
i:\;\;\;(\liel_i,L_i\cap K)\to(\liel_i',L_i'\cap K')
$$
sends finite length modules to finite length modules and discretely decomposables to discretely decomposables.

Now we fix an exponent $b\geq 0$ and consider for any $(\lieg,K)$-module $X$ whose restriction along $i$ is discretely decomposable the following boundedness condition:
We have for any $1\leq i\leq s$ and any degree $q$ that the multiplicity of any character $\lambda\in{\liel_i'}^*$ in the $q$-th $\lieu_i'$-cohomology of $X$ satisfies
\begin{equation}
m_\lambda(H^q(\lieu_i';X))\;\leq\; c_X\cdot\!\!\!\!\prod_{\beta\in\Delta(\lieu_i',\liel')}\!\!\!\langle\lambda+\rho(\lieu_i'),\beta\rangle^b + d_X.
\label{eq:multiplicitybound}
\end{equation}
for some  constants $c_X, d_X\geq 0$ independent of $X$ and $\lambda$ and $i$.

Then the modules $X$ satisfying the bound \eqref{eq:multiplicitybound} generate a subgroup $K_b'$ of $K'$ and we are interested in the commutative diagrams
$$
\begin{CD}
K_b'@>\iota>> K_{\rm df}(\lieg',K')\\
@Vc_{\lieq_i} VV@Vc_{\lieq_i'} VV\\
C_{\lieq_i,\rm fl}(\liel_i,L_i\cap K)@>\iota>> C_{\lieq_i',\rm df}(\liel_i',L')[W_{\lieq_i/\lieq_i'}^{-1}]
\end{CD}
$$
By the structure of the elements generating the kernel of the localization we know that the bound \eqref{eq:multiplicitybound} limits the degree of the annihilating Weyl denominator contributing to the kernel of the localization map in the image $\iota(K_b')$. The knowledge of the bound for the degree together with Theorem \ref{thm:kernel} yields the condition
\begin{itemize}
\item[(S')] 
For irreducible $Z$ with infinitesimal character $\lambda$ occuring in the restriction of $X$ to $(\lieg',K')$, for any $1\leq i\leq s$ and for any $w\in W(\lieg',\liel_i')$, and any numbering $\beta_1,\dots,\beta_r$ of the elements set $\Delta(\lieu_i,\liel_i')$ and any non-negative integers $n_1,\dots,n_r$ satisfying
\begin{equation}
n_1+\cdots+n_r\;=\;b\cdot\#\{\beta\in\Delta(\lieu_i',\liel_i')\mid\exists j:\langle\beta_j,\beta\rangle\neq 0\},
\label{eq:nsumb}
\end{equation}
the condition
\begin{equation}
|\langle w(\lambda)-\lambda_0,\beta_1\rangle|
\;\geq\;
\frac{1}{2}
\langle \beta_1,\beta_1\rangle
+
\sum_{l=1}^r
\frac{n_l+1}{2}
\langle \beta_1,\beta_l\rangle
\label{eq:lambdaregularityfd}
\end{equation}
holds.
\end{itemize}
Then the subgroup $K_{b,S'}'\subseteq K'$ generated by those $X$ satisfying condition (S') has trivial intersection with the kernel of the localization map corresponding to the pair $(\lieq_i, \lieq_i')$, and we obtain
\begin{theorem}[Januszewski, {\cite{januszewskipreprint}}]
For any $X\in K_{b,S'}'$ the multiplicities of all composition factors of $\iota(X)$ are uniquely determined by the images $\iota(c_{\lieq_i}(X))$ of the characters in $C_{\lieq_i',\rm df}(\liel_i',L_i')[W_{\lieq_i/\lieq_i'}^{-1}]$ via its simultaneous preimage of the maps $c_{\lieq_i'}$.
\end{theorem}

Here again we may generalize our result to modules possibly violating condition (S') as long as we know the contribution of the composition factors violating (S') for the various parabolics $\lieq_i$. We remark however that even in the compact case the analogous condition (S) may involve infinitely many irreducibles. This situation already arises if the root system of $\liek$ contains orthogonal roots.

So far we ignored the action of the Weyl group, which in certain cases eventually may allow us to strengthen condition (S) resp.\ (S').

Another remark is that the condition of {\em integrality} on the multiplicities, i.e.\ the fact that all $m_Z(X)$ are {\em integers} may be exploited in certain cases as well, especially in conjunction with multiplicity-one statements.

In general so far not much is known about multiplicities in discretely decomposable restrictions beyond Harish-Chandra's bound for the multiplicities of $K$-types in finite length representations.

We have the following
\begin{conjecture}[Kobayashi, {\cite[Conjecture C]{kobayashi2000}}]\label{conj:kobayashi}
Let $(G,G')$ be a semisimple symmetric pair, and $\pi\in\hat{G}$ an irreducible unitary representation of $G$. Assume that the restriction of $\pi$ to $G'$ is infinitesimally discretely decomposable, then the dimension
$$
\dim\Hom_{G'}(\tau,\pi|_{G'}),\;\;\;\tau\in\hat{G}'
$$
is finite.
\end{conjecture}

Motivated by our results presented in this section the author formulated
\begin{conjecture}[Januszewski, {\cite{januszewskipreprint}}]\label{conj:polybound}
In the setting of Conjecture \ref{conj:kobayashi} the dimension
$$
\dim\Hom_{G'}(\tau,\pi|_{G'}),\;\;\;\tau\in\hat{G}'
$$
grows at most polynomially in the norm of the infinitesimal character of $\tau$ (in the sense of \eqref{eq:multiplicitybound}).
\end{conjecture}

\section{Galois-equivariant characters}

In this section we sketch the first steps towards a Galois-equivariant algebraic character theory. The correct formal setup for this situation would be schemes representing representations, i.e.\ schemes parametrizing representations over the various extensions (including algebras, not only fields, and to allow for non-affine bases as well). However in order to keep the exposition as simple and self-contained as possible, we use an ad hoc approach in the classical language of linear algebraic groups and rational representations over fields of characteristic $0$. This fits well into Michael Harris' framework \cite{harris2012}.

Instead of relying on Beilinson-Bernstein-localization we introduce a rational version of Zuckerman's cohomological induction. This allows for a purely rational theory over any field $k$ of characteristic $0$, without the need to fall back to a universal domain like $\CC$, although we use $\CC$ in our exposition for convenience.

We define modules for pairs $(\liea_k,B_{k'})$ where $\liea_k$ is a Lie algebra defined over a field $k$ and $B_{k'}$ a reductive linear algebraic group defined over an extension $k'/k$. Such a module can, but need not have $k$-rational points. This makes the situation interesting, and opens many questions that we unfortunately cannot deal with here.

\subsection{Rationality of Harish-Chandra modules and Shimura varieties}

In the context of Shimura varieties considered in \cite{harris2012} the point of departure is a reductive algebraic group $G$ defined over $\QQ$, together with a Shimura datum $(G,X)$. The latter canonically defines a family of complex algebraic varieties $S_U(G,X)$ where $U$ runs through compact open subgroups of $G(\Adeles^{(\infty)})$ of the finite adelization of $G$. Following Shimura and Deligne the varieties $S_U(G,X)$ have canonical models over a number field $E(G,X)$. This rational structure comes from a set of distinguished points in $S_U(G,X)$, which stem from abelian varieties with complex multiplication. The rational structure on the latter has a description via certain reciprocity laws.

Now to any special point $x\in X$ corresponds a stabilizer $K_x\subseteq G(\RR)$, which is a compact subgroup of maximal dimension (but not necessarily maximal). However as the point $x$ varies under the action of an $\alpha\in\Aut(\CC)$, the group $K_x$ varies as well, and in order to control the rational structure one needs to pass from $G$ to an inner form ${}^{\alpha,x}G$ of $G$, as $K_x$ is not necessarily defined over $\QQ$, i.e. compactness of the real-valued points is not necessarily preserved. This complicates the study of rationality of Harish-Chandra modules, as it forces the study of a collections of pairs in order to lay hands on the Galois action in general. In \cite{harris2012} the fundamentals of such a theory is developped.

We follow a purely representation-theoretic approach here. For an introduction to the theory of linear algebraic groups, we refer to Mahir's lecture, and also to Borel's book \cite{book_borel1991}, \cite{book_chevalley2005}. The articles \cite{springer1979} and \cite{murnaghan2003} contain short introductions to the relative theory.

\subsection{Rational pairs and rational modules}

Let $k\subseteq\CC$ be a subfield. A {\em pair over $k$} is a pair $(\liea_k,B_{k'})$ consisting of a Lie algebra $\liea_k$ over $k$ and a reductive linear algebraic group $B_{k'}$ over an extension $k'/k$, a $k'$-rational inclusion
$$
i:\Lie(B_{k'})\to\liea_{k'}
$$
of the algebraic Lie algebra of $B_{k'}$ into
$$
\liea_{k'}\;:=\;\liea_k\otimes k',
$$
together with an extension of the adjoint action of $B_{k'}$ on $\Lie(B_{k'})$ to $\liea_{k'}$ whose differential is the adjoint action of $\Lie(B_{k'})$ as subalgebra on $\liea_{k'}$.

We write $\mathcal C_{\rm fd}(B_{k'})$ for the category of $k'$-rational finite-dimensional representations of $B_{k'}$. I.e.\ finite-dimensional $k'$-vectorspaces with a rational action of $B_{k'}(k')$. Then $\liea_{k'}$ is natually an object inside $\mathcal C_{\rm fd}(B_{k'})$ and the above conditions on the pair $(\liea_k,B_{k'})$  guarantee that $\liea_{k'}$ is a {\em Lie algebra object} inside this category. In other words, it comes with a $B_{k'}$-linear morphism
$$
[\cdot,\cdot]:\;\;\;\liea_{k'}\otimes_{k'}\liea_{k'}\;\to\;\liea_{k'}
$$
satisfying
$$
[\cdot,\cdot]\circ\Delta\;=\;0,
$$
where
$$
\Delta:\;\;\;\liea_{k'}\;\to\;\liea_{k'}\otimes_{k'}\liea_{k'}
$$
is the diagonal, and furthermore $[\cdot,\cdot]$ satisfies the Jacobi identity, which may be equally expressed diagrammatically inside of $\mathcal C_{\rm fd}(B_{k'})$.

Then a {\em finite-dimensional ($k'$-rational) $(\liea_{k'},B_{k'})$-module} is an $\liea_{k'}$-module object $M$ in $\mathcal C_{\rm fd}(B_{k'})$. In other words we have a map
$$
\liea_{k'}\otimes_{k'} M\;\to\;M
$$
in $\mathcal C_{\rm fd}(B_{k'})$ satisfying the usual axioms of a module over $\liea_{k'}$, which may be expressed diagrammatically as well. This gives us the category $\mathcal C_{\rm fd}(\liea_{k'},B_{k'})$ of finite-dimensional $(\liea_{k'},B_{k'})$-modules.

We introduce the category $\mathcal C(B_{k'})$ of ${\rm ind}$-objects in $\mathcal C_{\rm fd}(B_{k'})$. Every object in this category may be thought of as a (filtered) inductive limit of finite-dimensional representations of $B_{k'}$, or more explicitly as an increasing union of the latter.

Then $\liea_{k'}$ is again a Lie algebra object in $\mathcal C(B_{k'})$ and we may analogously to $\mathcal C_{\rm fd}(\liea_{k'},B_{k'})$ define the category $\mathcal C(\liea_{k'},B_{k'})$ of {\em $(\liea_{k'},B_{k'})$-modules rational over ${k'}$} as the category of $\liea_{k'}$-objects in $\mathcal C(B_{k'})$.

To a rational pair every inclusion $\sigma:{k'}\to\CC$ yields a complex Lie algebra
$$
\liea^\sigma\;:=\;\liea_{k'}\otimes_{{k'},\sigma}\CC
$$
and a complex Lie group
$$
B^\sigma(\CC)\;:=\;(B_{k'}\otimes_{{k'},\sigma}\CC)(\CC)
$$
Then inside this group we find a compact subgroup $B^\sigma$ of maximal dimension, unique up to conjugation by an element of $B^\sigma(\CC)$, with the property that the inclusion
$$
B^\sigma\;\to\;B^\sigma(\CC)
$$
induces an equivalence
$$
\mathcal C_{\rm fd}(B\otimes_{{k'},\sigma}\CC)\;\to\;
\mathcal C_{\rm fd}(B^\sigma)
$$
on the spaces of finite-dimensional representations. We call the pair (in the sense of section 2) $(\liea^\sigma,B^\sigma)$ {\em associated} to $(\liea_{k'},B_{k'})$.

So far we only defined $k'$-rational modules. We postpone the general definition it relies on the notion of base change.

\subsection{Rational models of reductive pairs}

Let $k\subseteq\CC$ be a subfield. A {\em reductive pair} over $k$ is a triple $(\lieg_k,K,G_k)$ consisting of a reductive linear algebraic group $G_k$ defined over $k$, its Lie algebra $\lieg_k$, and a reductive subgroup $K\subseteq G_k\otimes_k\CC$ with the property that there is an involution $\theta\in\Aut(G_k\otimes_k\CC)$ such that $K$ is the kernel of $\theta$.

We do not assume $K$ (or $\theta$) to be defined over $k$. As a subgroup of $G:=G_k\otimes_k\CC$ it acquires a natural rational structure and there is a unique field of definition $k_K/k$ inside $\CC$ and a unique reductive $k_K$-subgroup $K_{k_K}\subseteq G_k\otimes_k k_K$ whose base change to $\CC$ gives $K$.

We assume $\lieg_k\otimes_k\CC$ furthermore to be multiplicity-free as a $K$-module. This implies that $\liek_{k_K}$ has a unique $K_{k_K}$-invariant complement $\liep_{k_K}$ inside $\lieg_{k_K}$ and thus we have a canonical {\em Cartan decomposition}
$$
\lieg_{k_K}\;=\;\liep_{k_K}\;+\;\liek_{k_K},
$$
and in particular the involution $\theta$ is defined over $k_K$ and unique. To emphasize our idea of a {\em pair}, we write $(\lieg_k,K)^G$ for the triple, and call it {\em $K$-split} if $k_K=k$ and {\em non-$K$-split} otherwise.

To such a datum, split or non-split, every inclusion $\sigma:k_K\to\CC$ provides us, as above, with a complex reductive Lie algebra
$$
\lieg^\sigma\;:=\;\lieg_k\otimes_{k,\sigma}\CC
$$
and a complex Lie group
$$
K^\sigma(\CC)\;:=\;(K_{k_K}\otimes_{k_K,\sigma}\CC)(\CC).
$$
Again we find a compact subgroup $K^\sigma$ inside the latter which gives rise to equivalence between finite-dimsional rational representations of $(K_{k_K}\otimes_{k,\sigma}\CC$ and $K^\sigma$, and we call the pair $(\lieg^\sigma,K^\sigma)$ {\em associated} to $(\lieg_k,K)^G$.

If $(\lieg^\sigma,K^\sigma)$ is a reductive pair, then we say that $(\lieg_k\otimes_{k,\sigma}\sigma(k),K_k\otimes_k\sigma(k))$ resp.\ $(\lieg_k,\theta)^G$ is a {\em $\sigma(k)$-rational model}.

As in the previous section we may consider the category $\mathcal C_{{\rm fd}}(\lieg_{k_K},K_{k_K})$ of finite-dimensional ${k_K}$-rational $(\lieg_{k_K},K_{k_K})$-modules. It comes with a distinguished subcategory $\mathcal C_{{\rm fd}}(\lieg_{k_K},K_{k_K})^G$ which consists of the finite-dimensional ${k_K}$-rational rational representations of $G_{k_K}$.

This category is of particular importance to us, as it is eventually defined over $k$: the category $\mathcal C_{{\rm fd}}(\lieg_{k},K_{k})^G$ of $k$-rational rational (finite-dimensional) $G_k$-representations is a neutralized Tannakian category, i.e.\ it is a $k$-linear tensor category with a canonical fibre functor $\mathscr F$ into the category of finite-dimensional $k$-vector spaces. The linear algebraic group $G_k$ over $k$ as the Tannaka dual of $\mathcal C_{{\rm fd}}(\lieg_k,K_k)^G$. Then the forgetful functor
$$
\mathcal C_{{\rm fd}}(\lieg_{k_K},K_{k_K})^G\;\to\;
\mathcal C_{{\rm fd}}(\liek_{k_K},K_{k_K})^K
$$
corresponds to the inclusion $K_{K_k}\to G_{k_K}:=G_k\otimes_k k_K$ and thus to the induced inclusion of pairs.


\subsection{Base change and restriction of scalars}

For every map $\tau:k\to l$ of fields inside $\CC$ we get for every pair $(\liea_k,B_{k'})$ an $l$-rational pair $(\liea_l^\tau,B_{l'}^\tau)$ given by
$$
\liea_l^\tau\;:=\;\liea_k\otimes_{k,\tau}l,
$$
and
$$
B_{l'}^\tau\;:=\;B_{k'}\otimes_{k,\tau}l',
$$
where $l':=l\cdot\tau(l)$ denotes the compositum. Along the same lines every $(\liea_{k'},B_{k'})$-module $M_{k'}$ gives rise to an $(\liea_{l'}^\tau,B_{l'}^\tau)$-module
$$
M_{l'}^\tau\;:=\;M_{k'}\otimes_{k',\tau}l'.
$$
In the case where $\tau$ is a set-theoretic inclusion we drop it from the notation.

If $l$ is finite over $\tau(k)$ and if the degree of the extension $l/k$ is the same as the degree $l'/k'$, then this operation is left adjoint to {\em restriction of scalars} (in the sense of Weil \cite{book_weil1961}), which is defined as follows. Departing from an $l$-rational pair $(\liea_l,B_{l'})$ we obtain a $k$-rational pair
$$
(\res_\tau\liea_l,\res_\tau B_{l'})
$$
where $\res_\tau\liea_l$ denotes the pullback of the Lie algebra $\liea_l$ along $\tau:k\to l$, and $\res_\tau B_{l'}$ is the restriction of scalars of $B_{l'}$ along $\tau:k'\to l'$ in the sense of Weil. Then for any $(\liea_{l'},B_{l'})$-module $M_{l'}$ we get an $(\res_\tau\liea_{l'},\res_\tau B_{l'})$-module
$$
\res_\tau M_{l'}\;:=\;\tau^*(M_{l'}),
$$
again as the pullback along $\tau$. If $(\liea_{l'},B_{l'})=(\lieg_{l'},K)^G$ is a split reductive pair, then the restriction of scalars of $G_{l'}$ gives the reductive group associated to the restricted pair.

\subsection{Rational Harish-Chandra modules for non-$K$-split pairs}

We already introduced the category of $(\lieg_k,K)^G$-modules for $K$-split reductive pairs. Now in order to define a category of Harish-Chandra modules for non-$K$-split pairs we restrict ourselves to the connected case, i.e.\ we assume that $(\lieg_k,K)^G$ is a reductive pair with the property that $K$ is (geometrically) connected.

For any field extension $l/k$ any $l$-rational $\lieg_l$-module $M_l$ becomes a $\liek_{l'}$-module $M_{l'}$ after base change to $l':=l\cdot k_K$. We say that $M_l$ is a {\em $(\lieg_l,K)^G$-module} if it has the following properties:
\begin{itemize}
\item[($H_2^l$)] $M_{l'}$ is locally $\liek_{l'}$-finite.
\item[($H_3^l$)] the action of $\liek_{l'}$ on $M_{l'}$ lifts to $K_{l'}$.
\item[($H_1^l$)] the lifted action of $K_{l'}$ on $M_{l'}$ is compatible with the action of $\lieg_{l'}$.
\end{itemize}
We remark that the action in $(H_3^l)$ is unique if it exists as $K$ is connected and we are eventually only lifting finite-dimensional representations.

Mutatis mutandis we define $(\liea_l,B_{l'})$-modules whenever $B_{l'}$ is geometrically connected.

We call a map $f:M_l\to N_l$ between $(\lieg_l,K)^G$-modules is a {\em morphism} if it is $\lieg_l$-linear. This is equivalent to the base change $f:M_{l'}\to N_{l'}$ to be a morphism of $(\lieg_{l'},K_{l'})$-modules. This provides us with the category ${\mathcal C}(\lieg_l,K)$ of $(\lieg_l,K)^G$-modules. It contains the category $\mathcal C_{{\rm fd}}(\lieg_l,K)^G$ as a full abelian subcategory, and if $(\lieg_l,K)^G$ is $K$-split, it agrees with the previously defined category of $(\lieg_l,K)^G$-modules as $l=l'$ in this case.

We remark that as $K$ is connected, Schur's Lemma holds for simple modules in $\mathcal C(\lieg_l,K)$ by \cite{quillen1969}, and even in $\mathcal C(\liea_k,B_{k'})$ whenever $B_{k'}$ is geometrically connected.

We refer to \cite{lepowsky1976} for fundamental theorems on linear properties of rational reductive pairs.

\subsection{Equivariant cohomology}

Let $(\liea_k,B_k)$ be any $k$-rational pair with $B_k$ defined over $k$ as well. As $B_k$ is reductive, the categories of finite-dimensional representations of $B_k$ is semisimple, hence all objects in $\mathcal C_{\rm fd}(B_k)$ are injective and projective, and the same remains true in the ind-category $\mathcal C(B_k)$.

The forgetful functor $\mathcal F_{B_k}^{\liea_k,B_k}$ sending $(\liea_k,B_k)$-modules to $B_k$-modules has a left adjoint
$$
\ind_{B_k}^{\liea_k,B_k}:\;\;\;M\;\mapsto\; U(\liea_k)\otimes_{U(\lieb_k)}M
$$
and a right adjoint
$$
\pro_{B_k}^{\liea_k,B_k}:\;\;\;M\;\mapsto\; \Hom_{\lieb_k}(U(\liea_k),M)_{B_k-\text{finite}}.
$$
Then $\ind_{B_k}^{\liea_k,B_k}$ sends projectives to projectives and $\pro_{B_k}^{\liea_k,B_k}$ sends injectives to injectives. As all $(\lieb_k,B_k)$-modules are injective and projective we see that $\mathcal C(\liea_k,B_k)$ has enough injectives and enough projectives. Therefore standard methods from homological algebra apply, and we may define cohomology as follows.

Pick a $K$-split reductive pair $(\lieg_k,K)^G$ over $k$ and a $k$-rational parabolic subalgebra $\lieq_k$ which is assumed to have a $k$-rational Levi decomposition
$$
\lieq_k\;=\;\liel_k+\lieu_k.
$$
We call $\lieq_k$ {\em germane (over $k$)} if $\liel_k$ is the Lie algebra of a reductive subgroup $L_k$ of $G_k$ defined as follows.

Consider the category $\mathcal C_{\rm fds}(\liel_k)$ of semi-simple finite-dimensional $\liel_k$-modules. This is a neutralized Tannakian category with Tannaka dual $\mathcal L_k$ say. Then $\mathcal L_k$ is proreductive and the forgetful functor

$$
\mathcal C_{\rm fd}(G_k)\;\to\;
\mathcal C_{\rm fds}(\liel_k)
$$
defines a morphism
$$
\ell:\mathcal L_k
\;\to\;
G_k.
$$
We then define $L_k$ to be the image of $\ell$. It is defined over $k$. If the Lie algebra of $L_k$ is $\liel_k$ we call $\lieq_k$ germane.

In this case $(\lieq_k,L_k\cap K_k)$ a {\em $k$-rational parabolic subpair} of $(\lieg_k,K)^G$. We have a map
$$
p:\;\;\;(\lieq_k,L_k\cap K_k)\;\to\;(\liel_k,L_k\cap K_k)
$$
of pairs and pullback along $p$ gives a functor
$$
\mathcal C(\liel_k,L_k\cap K_k)\to \mathcal C(\lieq_k,L_k\cap K_k).
$$
It has a right adjoint
$$
H^0(\lieu_k;-):\;\;\;\mathcal C(\lieq_k,L_k\cap K_k)\to \mathcal C(\liel_k,L_k\cap K_k)
$$
given by taking $\lieu_k$-invariants. The higher right derived functors
$$
H^q(\lieu_k;-):=R^qH^0(\lieu_k;-):\;\;\;\mathcal C(\lieq_k,L_k\cap K_k)\to \mathcal C(\liel_k,L_k\cap K_k)
$$
are the {\em $k$-rational $\lieu_k$-cohomology} and may be computed via the usual standard complex. Dually we may define and explicitly compute $\lieu_k$-homology. These homology and cohomology theories satisfy the usual duality relations.

\begin{proposition}\label{prop:equicohomology}
For any $k$-rational $(\lieg_k,K_k)$-module $M_k$ the cohmology $H^q(\lieu_k,X_k)$ is $k$-rational and for any map $\tau:k\to l$ of fields we have a natural isomorphism
$$
H^q(\lieu_k,X_k)_l^\tau\;\to\;
H^q(\lieu_l^\tau,X_l^\tau)
$$
of $(\liel_l^\tau,L_l^\tau\cap K_l^\tau)$-modules. The same statement is true for $\lieu_k$-homology and the duality maps and Hochschild-Serre spectral sequences respect the rational structure.
\end{proposition}

\begin{proof}
This is obvious from the $k$-resp.\ $l$-rational standard complexes computing cohomology and homology.
\end{proof}

Along the same lines we may define rational $(\lieg_k,K_k)$-cohomology. It also commutes with base change.


\subsection{Equivariant cohomological induction}

To define $k$-rational cohomological induction, we adapt Zuckerman's original construction as in \cite[Chaper 6]{book_vogan1981}. We start with a $(\lieg_k,L_k\cap K_k)$-module $M$ and set
$$
\tilde{\Gamma}_0(M)\;:=\;
\;\{m\in M\mid\dim_k U(\liek_k)\cdot m\;<\;\infty\}.
$$
and
$$
\Gamma_0(M)\;:=\;
\;\{m\in \tilde{\Gamma}_0(M)\mid \text{the $\liek_k$-representation}\; U(\liek_k)\cdot m\;\text{lifts to}\;K_k^0\}.
$$
As in the analytic case the obstruction for a lift to exist is the (algebraic) fundamental group of $K_k^0$. In particular there is no rationality obstruction, as a representation $N$ of $\liek_k$ lifts to $K_k^0$ if and only it does so after base change to one (and hence any) extension of $k$. It is easy to see that $\Gamma_0(M)$ is a $(\lieg_k,K_k^0)$-module.

We define the space of {\em $K_k^0$-finite vectors} in $M$ as
$$
\Gamma_1(M)\;:=\;
$$
$$
\;\{m\in \Gamma_0(M)\mid \text{the actions of $L_k\cap K_k$ (on $M$) and $L_k\cap K_k^0\subseteq K_k^0$  agree on $m$}\}.
$$
This is a $(\lieg_k,(L_k\cap K_k)\cdot K_k^0)$-module.
Finally the space of {\em $K_k$-finite vectors} in $M$ is
$$
\Gamma(M)\;:=\;
\;\ind_{\lieg_k,(L_k\cap K_k)\cdot K_k^0}^{\lieg_k,K_k}(\Gamma_1(M)).
$$
However this is not a subspace of $M$ in general. But by Frobenius reciprocity it comes with a natural map $\Gamma(M)\to M$.

The functor $\Gamma$ is a right adjoint to the forgetful functor along
$$
i:\;\;\;(\lieg_k,L_k\cap K_k)\;\to\;(\lieg_k,K_k)
$$
and hence sends injectives to injectives. We obtain the higher Zuckerman functors as the right derived functors
$$
\Gamma^q\;:=\;R^q\Gamma:\;\;\;
\mathcal C(\lieg_k,L_k\cap K_k)\to \mathcal C(\lieg_k,K_k).
$$
As in the classical case we can show
\begin{proposition}\label{prop:find}
For each $q$ we have a commutative square
$$
\begin{CD}
\mathcal C(\lieg_k,L_k\cap K_k)@>\Gamma^q>>\mathcal C(\lieg_k,K_k)\\
@V\mathcal F_{\liek_k,L_k\cap K_k}^{\lieg_k,L_k\cap K_k}VV
@VV\mathcal F_{\liek_k,K_k}^{\lieg_k,K_k}V\\
\mathcal C(\liek_k,L_k\cap K_k)@>\Gamma^q>>\mathcal C(\liek_k,K_k)\\
\end{CD}
$$
\end{proposition}

\begin{proof}
  For $q=0$ the commutativity is obvious from the explicit construction of the functor $\Gamma$. For $q>0$ this follows from the standard argument that the forgetful functors have an exact left adjoint given by induction along the Lie algebras and hence carry injectives to injectives. Furthermore they are exact, which means that the Grothendieck spectral sequences for the two compositions both degenerate. Therefore edge morphisms yield the commutativity, this proves the claim.
\end{proof}

\begin{corollary}\label{cor:basechange}
The functors $\Gamma^q:\mathcal C(\lieg_k,L_k\cap K_k)\to\mathcal C(\lieg_k,K_k)$ commute with base change and for any map $\tau:k\to l$ of fields and any $(\lieg_k,L_k\cap K_k)$-module $M_k$ we have for any degree $q$ a natural isomorphism
$$
\left(\Gamma^q(M_k)\right)_{l}^\tau\;\to\;
\Gamma^q(M_l^\tau).
$$
\end{corollary}

\begin{proof}
By universality of base change we obtain a natural map
$$
\left(\Gamma^q(M_k)\right)_{l}^\tau\;\to\;
\Gamma^q(M_l^\tau),
$$
which is compatible with the commutative diagram in Proposition \ref{prop:find}. This reduces us to the case $\lieg_k=\liek_k$, i.e.\ the finite-dimensional case where the statement is well known.
\end{proof}

In particular the functors $\Gamma^q$ satisfy the usual properties, i.e.\ they vanish for $q > \dim_k\liek_k/\liek_k\cap \liel_k$, we have a Hochschild-Serre spectral sequence for $K_k$-types, the effect on infinitesimal characters is the same as in the classical setting, etc.

Let $Z$ be an $(\liel_k,L_k\cap K_k)$-module. Then we may consider it as a $(\lieq_k,L_k\cap K_k)$-module with trivial $\lieu_k$-action. From there we consider
$$
\mathcal R^q(Z)\;:=\;
\Gamma^q\pro_{\lieq_k,L_k\cap K_k}^{\lieg_k,L_k\cap K_k}(Z\otimes_k\bigwedge^{\dim_k\lieu_k}\lieu_k),
$$
where
$$
\pro_{\lieq_k,L_k\cap K_k}^{\lieg_k,L_k\cap K_k}(Z\otimes_k\bigwedge^{\dim_k\lieu_k}\lieu_k)\;:=\;
\Hom_{\lieq_k}(U(\lieg_k),Z\otimes_k\bigwedge^{\dim_k\lieu_k}\lieu_k
)_{L_k\cap K_k-\text{finite}}.
$$

The definition of the functor $\Gamma_0(M)$ may be extended to non-$K$-split reductive pairs via the maximal submodule of $M$ which maps into $\Gamma_0(M\otimes_k {k_K})$. However it is not so clear how to generalize the remaining functors to this setting.

\subsection{Rational models of Harish-Chandra modules}

We depart from a $k$-rational pair $(\liea_k,B_k)$. Let $\sigma:k\to\CC$ be an embedding, and $(\liea^\sigma,B^\sigma)$ be an associated pair. Then
\begin{proposition}\label{prop:classicaliso}
The category $\mathcal C(\liea_\CC^\sigma,B_\CC^\sigma)$ of $\CC$-rational modules is naturally equivalent to the category $\mathcal C(\liea^\sigma,B^\sigma)$ of classical $(\liea^\sigma,B^\sigma)$-modules. This equivalence induces an equivalence of the corresponding categories of finite-dimensional modules.
\end{proposition}

\begin{proof}
By construction the categories are equivalent in the case $\liea_k=\lieb_k$, and the general case follows from this observation as well.
\end{proof}

Now by the classification of reductive algebraic groups we know that each reductive pair $(\lieg,K)$ has a $k$-rational model $(\lieg_k,K_k)$ over a {\em number field} $k\subseteq\CC$, i.e.\ $k/\QQ$ is finite (we may assume $K_k$ quasi-split upon replacing $k$ by a finite extension). However in general such a model is far from unique.

\begin{corollary}
Each cohomologically induced $(\lieg,K)$-module has a model over the field of definition if its inducing data.
\end{corollary}

In particular if the parabolic subpair $(\lieq_k,L_k\cap K_k)$, its Levi decomposition $\liel_k+\lieu_k$, and the inducing module $Z$ are all defined over $k\subseteq\CC$, then so are the induced modules $\mathcal R^q(Z)$ for all $q$. In particular every the $(\lieg,K)$-module of any discrete series representation, or more generally of any unitary representation with non-trivial $(\lieg,K)$-cohomology and $\overline{\QQ}$-rational infinitesimal character has a model over a {\em number field}.

The classical character formulae, as for example in the latter case given in \cite{voganzuckerman1984}, are rational over the same field of definition as well if interpreted in our theory as follows.

The standard arguments proving (ii) and (iii) of Theorem \ref{thm:inheritance} carry over to our setting provided that $\lieq_k$ is {\em $\theta$-stable} in the algebraic sense: Write $\liep_k$ for the $K_k$-complement of $\liek_k$ in $\lieg_k$, then
\begin{equation}
\lieq_k\;=\;(\lieq_k\cap\liep_k)+(\lieq_k\cap\liek_k).
\label{eq:rationalthetastable}
\end{equation}
\begin{proposition}\label{prop:rationalfl}
Assume that $(\lieg_k,K)^G$ is a reductive pair and that $\lieq_k\subseteq\lieg_k$ is a $k$-germane parabolic subalgebra. Then the $k$-rational $\lieu_k$-cohomology preserves $Z(\lieg_k)$-finiteness and if $\lieq_k$ is $\theta$-stable in the sense of \eqref{eq:rationalthetastable} also admissbility, in particular it sends finite length modules to finite length modules in that case.
\end{proposition}

\begin{proof}
Without loss of generality we may enlarge $k$ by Proposition \ref{prop:equicohomology} and in particular we may assume $G_k$ to be split. Then the Harish-Chandra map is defined over $k$. The rest of the argument goes as in the classical case, cf.\ \cite[Theorems 7.56 and Corollary 5.140]{book_knappvogan1995}.
\end{proof}

Proposition \ref{prop:rationalfl} allows us to define $k$-rational algebraic characters for finite length modules using the same formalism mutatis as before:
$$
c_{\lieq_k}:\;K_{\rm fl}(\lieg_k,K_k)\;\to\;
C_{\lieq_k,\rm fl}(\liel_k,L_k\cap K_k):=K_{\rm fl}(\liel_k,L_k\cap K_k)[W_{\lieq_k}^{-1}],
$$
$$
M_k\;\mapsto\;
\frac{\sum_q(-1)^q[H^q(\lieu_k;M_k)]}
{\sum_q(-1)^q[H^q(\lieu_k;{\bf1}_k)]}.
$$
Then $c_{\lieq_k}$ is multiplicative, respects duals, commutes with base change, and is therefore compatible with our theory from section 5 over $\CC$.


$$
\underline{\;\;\;\;\;\;\;\;\;\;\;\;\;\;\;\;\;\;\;\;\;\;\;\;\;\;\;\;\;\;}
$$\ \\
Karlsruher Institut f\"ur Technologie, Fakult\"at f\"ur Mathematik, Institut f\"ur Algebra und Geometrie, Kaiserstra\ss{}e 89-93, 76133 Karlsruhe, Germany\\
{januszewski@kit.edu}

\end{document}